\RequirePackage{fix-cm}
\documentclass[smallextended]{svjour3}       
\smartqed  
\usepackage{graphicx}
\usepackage{amsmath,amssymb,txfonts}
\usepackage{color,bm,comment}
\def\re{\mathbb{R}}
\def\N{\mathbb{N}}

\def\s{\mathbb{S}^{N-1}}
\def\tc{\textcolor}
\def\({\left(}
\def\){\right)}
\def\[{\left[}
\def\]{\right]}
\def\pd{\partial}
\def\lap{\Delta}

\def\ep{\varepsilon}
\def\w{\omega}
\def\la{\lambda}

\newtheorem{ThmA}{Theorem A}

\begin{document}

\title{Improvements and generalizations of two Hardy type inequalities and their applications to the Rellich type inequalities 
}

\titlerunning{Improved Hardy and Rellich type inequalities}        

\author{Megumi Sano         
}


\institute{Megumi Sano \at
              Laboratory of Mathematics, School of Engineering,
Hiroshima University\\
Higashi-Hiroshima, 739-8527, Japan \\
              Tel.: +81-82-424-7149\\
              \email{smegumi@hiroshima-u.ac.jp}           
}

\date{Received: date / Accepted: date}

\maketitle

\begin{abstract}
We give improvements and generalizations of both the classical Hardy inequality and the geometric Hardy inequality based on the divergence theorem. Especially, our improved Hardy type inequality derives both two Hardy type inequalities with best constants. 
Besides, we improve two Rellich type inequalities by using the improved Hardy type inequality.

\keywords{The Hardy inequality \and Higher order \and Optimal constant  \and improvement}
\subclass{35A23 \and 26D10 \and 58E40}
\end{abstract}

\section{Introduction}\label{S intro}

One dimensional Hardy type inequality
\begin{align}\label{1dim H}
\( \frac{p-1-a}{p} \)^p \int_0^\infty  \( \frac{1}{x} \int_0^x f(t) \,dt \)^p x^a \,dx \le  \int_0^\infty f^p(x) x^a \,dx
\end{align}
holds for all measurable nonnegative function $f$, where $p>1$ and $a < p-1$ (Ref. \cite{H(1920),HLP}). Concerning a history of (\ref{1dim H}), see \cite{Kbook}. 
On the other hand, higher dimensional version of (\ref{1dim H}) are two main inequalities. One is the classical Hardy inequality with an interior singularity
\begin{equation}
\label{H_p}
\( \frac{N-\alpha}{p} \)^{p} \int_{B_R} \frac{|u|^{p}}{|x|^\alpha} dx \le \int_{B_R} \frac{\left| \nabla u\right|^p}{|x|^{\alpha-p}} dx
\end{equation}
for $u \in C_c^1 (B_R)$, where $B_R \subset \re^N, N \ge 2, 1 < p < \infty$, and $\alpha < N$. 
Especially, in the case where $\alpha =p$, it is known that Hardy's best constant $\( \frac{N-p}{p} \)^p$ plays an important role to investigate several properties of solution to elliptic and parabolic partial differential equations, for example, stability of solution, instantaneous blow-up solution, global-in-time solution, see \cite{BV,BG}, to name a few. 
The other is the geometric Hardy type inequality with a boundary singularity
\begin{align}\label{H_p geo}
\(  \frac{\beta -1}{p}  \)^{p} \int_{B_R} \frac{|u|^p}{ {\rm dist}(x, \pd B_R)^{\,\beta}} \,dx
\le \int_{B_R} \frac{|\nabla u|^p}{ {\rm dist}(x, \pd B_R)^{\,\beta -p}} \,dx.
\end{align}
for $u \in C_c^1 (B_R)$, where $1< p< \infty, \beta >1$, and dist $(x, \pd B_R) = R- |x|$. The inequality (\ref{H_p geo}) also holds for general bounded domain, and its best constant depends on the geometry of the domain, see \cite{BM,BFT,BT}, to name a few. 


One of aims is to combine two Hardy type inequalities (\ref{H_p}), (\ref{H_p geo}), namely, to give the $N-$dimensional Hardy type inequality which derives both two Hardy type inequalities (\ref{H_p}), (\ref{H_p geo}) with best constants.
The first result is as follows.


\begin{theorem}\label{T IH}
Let $1 \le p < \infty, 1< \beta < \infty, \gamma > 0$, and $\alpha \le N-(\beta -1)\gamma$. Then the inequality
\begin{equation}\label{IH gene}
\( \frac{\beta -1}{p} \gamma \)^{p} \int_{B_R} \frac{|u|^{p}}{|x|^\alpha \( 1- \( \frac{|x|}{R} \)^\gamma \)^\beta } dx \le \int_{B_R} \frac{\left| \nabla u \cdot \frac{x}{|x|} \right|^p}{|x|^{\alpha-p}\( 1-\( \frac{|x|}{R} \)^\gamma \)^{\beta-p}} dx
\end{equation}
holds for all $u \in C_c^1(B_R)$. Furthermore, the constant $(\frac{\beta-1}{p} \gamma)^p$ in (\ref{IH gene}) is optimal and is not attained for $u \not\equiv 0$ for which the right-hand side is finite.
\end{theorem}

We see that our inequality (\ref{IH gene}) derives several Hardy type inequalities  while keeping their best constants. 

\begin{corollary}\label{C IH first}
(I) Let $1 < p < \infty, \gamma = \frac{N-\alpha}{p-1} > 0$. Then the following inequalities hold for any functions $u \in C_{c}^{1}(B_R)$:
\begin{equation*}
\( \frac{N-\alpha}{p} \)^{p} \int_{B_R} \frac{|u|^{p}}{|x|^\alpha} dx  \le \( \frac{p -1}{p} \gamma \)^{p} \int_{B_R} \frac{|u|^{p}}{|x|^\alpha \( 1- \( \frac{|x|}{R} \)^\gamma \)^p } dx \le \int_{B_R} \frac{\left| \nabla u \cdot \frac{x}{|x|} \right|^p}{|x|^{\alpha-p}} \, dx
\end{equation*}

\noindent
(II) Let $1\le p \le N -(\beta -1)$ and $\beta >1$. Then the following inequalities hold for any functions $u \in C_{c}^{1}(B_R)$:
\begin{align*}
\(  \frac{\beta -1}{p}  \)^{p} \int_{B_R} \frac{|u|^p}{ {\rm dist}(x, \pd B_R)^{\,\beta}} \,dx
&\le \(  \frac{\beta -1}{p}  \)^{p} R^{p-\beta} \int_{B_R} \frac{|u|^p}{|x|^{p} \( 1- \frac{|x|}{R}  \)^{\beta}} \,dx \\
&\le R^{p-\beta}  \int_{B_R} \frac{\left| \nabla u \cdot \frac{x}{|x|} \right|^p}{\( 1- \frac{|x|}{R} \)^{\beta -p} }\,dx 
\le \int_{B_R} \frac{\left| \nabla u \cdot \frac{x}{|x|} \right|^p}{ {\rm dist}(x, \pd B_R)^{\,\beta -p}} \,dx
\end{align*}

\noindent
(III) Let $N \ge 3$. Then the following inequalities hold for any functions $u \in C_{c}^{1}(B_1)$:
\begin{align}\label{HSM ineq}
\int_{B_1} \frac{|u|^2}{(1-|x|^2)^2} \,dx
\le \int_{B_1} \frac{|u|^2}{|x|^{2} \( 1- |x|^2  \)^{2}} \,dx 
\le  \int_{B_1} \left| \nabla u \cdot \frac{x}{|x|} \right|^2 \,dx 
\end{align}
\end{corollary}

Hardy-Sobolev-Maz'ya inequality which is the inequality (\ref{HSM ineq}) with Sobolev type remainder term has been investigated, see e.g. \cite{Mbook} p.139, Corollary 3, \cite{BFL,TT}.

\begin{remark}
The inequality (\ref{IH gene}) holds for $\gamma \in (0, \frac{N-\alpha}{\beta -1}]$ and becomes a trivial form when $\gamma =0$, since its best constant $\( \frac{\beta -1}{p} \gamma \)^p$ is zero. 
However we can take a limit of (\ref{IH gene}) as $\gamma \to 0$ as follows,  since $1-r^x = x \log \frac{1}{r} + o(1)$ as $x \to 0$.
\begin{align*}
\( \frac{\beta -1}{p} \gamma \)^{p} \int_{B_R} \frac{|u|^{p}}{|x|^\alpha \( 1- \( \frac{|x|}{R} \)^\gamma \)^\beta } dx
&= \( \frac{\beta -1}{p}  \)^{p} \gamma^{p-\beta} \int_{B_R} \frac{|u|^{p}}{|x|^\alpha \( \log \frac{R}{|x|} \)^\beta } dx  + o(1), \\
\int_{B_R} \frac{\left| \nabla u \cdot \frac{x}{|x|} \right|^p}{|x|^{\alpha-p}\( 1-\( \frac{|x|}{R} \)^\gamma \)^{\beta-p}} dx
&= \gamma^{p-\beta} \int_{B_R} \frac{\left| \nabla u \cdot \frac{x}{|x|} \right|^p}{|x|^{\alpha -p} \( \log \frac{R}{|x|} \)^{\beta -p} } dx  + o(1)\,\, (\gamma \to 0).
\end{align*}
As a consequence, we can obtain the inequality
\begin{equation}\label{IH log}
\( \frac{\beta -1}{p} \)^{p} \int_{B_R} \frac{|u|^{p}}{|x|^\alpha \( \log \frac{R}{|x|} \)^\beta } dx \le \int_{B_R} \frac{\left| \nabla u \cdot \frac{x}{|x|} \right|^p}{|x|^{\alpha-p}\( \log \frac{R}{|x|} \)^{\beta-p}} dx.
\end{equation}
as a limiting form of the inequality (\ref{IH gene}) as $\gamma \to 0$, where $\beta >1, \alpha \le  N$, and $1 \le p < \infty$. Concerning  to the inequality (\ref{IH log}), see \cite{TF,MOW,RS}. 
\end{remark}

The special case of (\ref{IH gene}) where $\alpha = N-(\beta -1)\gamma$ and $\beta = p$ coincides with the inequality (1.20) in \cite{I} which is shown by harmonic transplantation from $B_1$ to $\re^N$. 
However, our approach is based on the divergence theorem unlike \cite{I}. Therefore, it is straighter and simpler. 
Besides, we give generalizations of two Hardy type inequalities (\ref{H_p}), (\ref{H_p geo}). These generalized inequalities follow naturally from the divergence theorem. 

\begin{theorem}\label{T IH another}
Let $1 \le p < \infty, \gamma > 0$ and $\alpha \le N$.

\noindent
(I) If $\beta \ge0$, then the inequality
\begin{equation}\label{IH gene another}
\( \frac{N-\alpha}{p} \)^{p} \int_{B_R} \frac{|u|^{p}}{|x|^\alpha \( 1- \( \frac{|x|}{R} \)^\gamma \)^\beta } dx \le \int_{B_R} \frac{\left| \nabla u \cdot \frac{x}{|x|} \right|^p}{|x|^{\alpha-p}\( 1-\( \frac{|x|}{R} \)^\gamma \)^{\beta}} dx
\end{equation}
holds for all $u \in C_c^1(B_R)$. Furthermore, the constant $(\frac{N-\alpha}{p} )^p$ in (\ref{IH gene another}) is optimal and is not attained for $u \not\equiv 0$ for which the right-hand side is finite.

\noindent
(II) Let $\beta >1$ and $\alpha <  N - (p-1) \gamma$. Then the inequality
\begin{equation}\label{IH gene another 2}
\( \frac{\beta -1}{p} \gamma \)^{p} \int_{B_R} \frac{|u|^{p}}{|x|^{\alpha -\gamma} \( 1- \( \frac{|x|}{R} \)^\gamma \)^{\beta} } dx \le \int_{B_R} \frac{\left| \nabla u \cdot \frac{x}{|x|} \right|^p}{|x|^{\alpha-\gamma + (\gamma -1) p}\( 1-\( \frac{|x|}{R} \)^\gamma \)^{\beta  -p}} dx
\end{equation}
holds for all $u \in C_c^1(B_R)$. Furthermore, the constant $(\frac{\beta-1}{p} \gamma)^p$ in (\ref{IH gene another 2}) is optimal and is not attained for $u \not\equiv 0$ for which the right-hand side is finite.
\end{theorem}

\begin{remark}\label{Rem another}
In Theorem \ref{T IH another} (I), the special case where $\beta = 0$ coincides with the classical Hardy type inequality (\ref{H_p}). On the other hand, in Theorem \ref{T IH another} (II), the special case where $R=\alpha = \gamma =1$ coincides with the geometric Hardy type inequality (\ref{H_p geo}).
\end{remark}

\begin{remark}
If we take limits of (\ref{IH gene another}) and (\ref{IH gene another 2}) as $\gamma \to 0$, then we obtain the following inequalities:
\begin{align*}
\( \frac{N-\alpha}{p} \)^{p} \int_{B_R} \frac{|u|^{p}}{|x|^\alpha \( \log \frac{R}{|x|} \)^\beta } dx &\le \int_{B_R} \frac{\left| \nabla u \cdot \frac{x}{|x|} \right|^p}{|x|^{\alpha-p}\( \log \frac{R}{|x|} \)^{\beta}} dx\\
\( \frac{\beta -1}{p} \)^{p} \int_{B_R} \frac{|u|^{p}}{|x|^\alpha \( \log \frac{R}{|x|} \)^\beta } dx &\le \int_{B_R} \frac{\left| \nabla u \cdot \frac{x}{|x|} \right|^p}{|x|^{\alpha-p}\( \log \frac{R}{|x|} \)^{\beta -p}} dx
\end{align*}
\end{remark}

\begin{remark}
It is well-known that each Hardy type inequality has each {\it virtual extremal}, which attains  the optimal constant, however, is not in the suitable functional space since the integral in the inequality diverges, see also \cite{CF,S(MIA)} for distance type remainder terms from the  virtual extremal. 
From the proof of Theorem \ref{T IH another}, the virtual extremal of the inequality (\ref{IH gene another}) is $|x|^{-\frac{N-\alpha}{p}}$, which is equal to $|x|^{-\frac{\beta -1}{p} \gamma}$ when $\gamma = \frac{N -\alpha}{\beta -1}$. On the other hand, the virtual extremal of the inequality (\ref{IH gene another 2}) is $(1-|x|^\gamma )^{\frac{\beta -1}{p}}$. 
We see that the mixed function $|x|^{-\frac{\beta -1}{p} \gamma} (1-|x|^\gamma )^{\frac{\beta -1}{p}} = (|x|^{-\gamma} -1 )^{\frac{\beta -1}{p}}$ becomes the virtual extremal of the inequality (\ref{IH gene}), see the proof of Theorem \ref{T IH} and Remark \ref{R virtual extremal}. 
\end{remark}


Theorem \ref{T IH} with $p=\alpha = \beta = 2$ is closely related to the following minimization problem and the eigenvalue problem with a singular potential $\frac{1}{|x|^2 (1-|x|^\gamma )^{2}}$
\begin{align}\label{EL}
\( \frac{\gamma}{2} \)^2 = \inf_{u \in H_0^1 (B_1) \setminus \{ 0\}} \frac{\int_{B_1} |\nabla u |^2 \,dx}{\int_{B_1} \frac{|u|^2}{|x|^2 (1-|x|^\gamma )^2} \,dx}, \quad
\begin{cases}
-\lap u = \( \frac{\gamma}{2} \)^2 \frac{u}{|x|^2 (1-|x|^\gamma )^{2}} \quad &\text{in} \,\, B_1, \\
\quad  \,u = 0  \,\, &\text{on} \,\, \pd B_1.
\end{cases}
\end{align}
Theorem \ref{T IH} implies that the above infimum is not attained, and a weak solution of (\ref{EL}) does not exist. 
Furthermore, the above infimum becomes the threshold of whether a solution of the heat equation with the above singular potential exists, see \cite{CabreM}.


This paper is organized as follows:
We show Theorem \ref{T IH} in subsection \ref{Sec 2.1} and show Theorem \ref{T IH another} in subsection \ref{Sec 2.3}. Both Theorems are proved based on the divergence theorem. In subsection \ref{Sec trans}, we also give another proof of Theorem \ref{T IH} based on transformation approach. We give a generalization of the harmonic transplantation from $B_1$ to $\re^N$ used in \cite{I}. 
By using this transformation, we study minimization problems associated with the  improved Hardy-Sobolev type inequalities. 
In section \ref{Sec Higher}, we also give an improvement of the Rellich type inequalities as an application of Theorem \ref{T IH}. 
We consider the case where $p=2$ for any functions in subsection \ref{S IR p=2}, and   the case where $p \not= 2$ for radially symmetric functions in subsection \ref{S IR rad}.  
In section \ref{Appendix}, we give one dimensional inequalities and calculations to show Theorems.

We fix several notations: 
$X_{{\rm rad}} = \{ \, u \in X \, | \, u \,\,\text{is radially symmetric} \, \}$. 
$B_R$ 
denotes a $N$-dimensional ball centered $0$ with radius $R$. 
$\omega_{N-1}$ denotes an area of the unit sphere $\mathbb{S}^{N-1}$ in $\re^N$. 





\section{First order inequalities}\label{S IH}

\subsection{The divergence theorem and improvements of two Hardy type inequalities: Proof of Theorem \ref{T IH}}\label{Sec 2.1}

We show the following two Theorems. 
Both Theorems are shown by the divergence theorem and imply Theorem \ref{T IH}. 

\begin{theorem}\label{T IH remainder}
Let $1 < p, \beta < \infty, \alpha \in \re, \gamma > 0$, and $\alpha \le N-(\beta -1)\gamma$. Then the inequality
\begin{equation}\label{IH remainder}
\( \frac{\beta -1}{p} \gamma \)^{p} \int_{B_R} \frac{|u|^{p}}{|x|^\alpha \( 1- \( \frac{|x|}{R} \)^\gamma \)^\beta } dx + \psi_{N,p,\alpha, \beta} (u) \le \int_{B_R} \frac{\left| \nabla u \cdot \frac{x}{|x|} \right|^p}{|x|^{\alpha-p}\( 1-\( \frac{|x|}{R} \)^\gamma \)^{\beta-p}} dx
\end{equation}
holds for all $u \in C_c^1(B_R)$, where $C >0$ depends on $p$ and $N$,
\begin{align*}
\psi_{N,p, \alpha, \beta} (u ) = (N-\alpha - (\beta -1) \gamma) \( \frac{\beta -1}{p} \gamma \)^{p-1}\int_{B_R} \frac{|u|^p}{|x|^{\alpha} \(1- \( \frac{|x|}{R} \)^\gamma \)^{\beta-1}} \,dx
\end{align*}
for $\alpha < N-(\beta -1)\gamma$, and 
\begin{align*}
\psi_{N,p,\alpha, \beta} (u)
=\begin{cases} 
C \int_{B_R} |x|^{p-N}\( 1- \( \frac{|x|}{R} \)^\gamma \)^{p-1} \left| \nabla \( \frac{u(x)}{\( \( \frac{|x|}{R} \)^{-\gamma} -1 \)^{\frac{\beta -1}{p}}} \) \cdot \frac{x}{|x|} \right|^p \,dx \,\,&\text{if}\,\, p \in [2, \infty), \\
C \( \int_{B_R} |x|^{p-N}\( 1- \( \frac{|x|}{R} \)^\gamma \)^{p-1} \left| \nabla \( \frac{u(x)}{\( \( \frac{|x|}{R} \)^{-\gamma} -1 \)^{\frac{\beta -1}{p}}} \) \cdot \frac{x}{|x|} \right|^p \,dx \)^{\frac{2}{p}} &\tc{white}{a}\\
\hspace{5em}\times \( \int_{B_1} \frac{\left| \nabla u \cdot \frac{x}{|x|} \right|^p}{|x|^{\alpha-p} \( 1-\( \frac{|x|}{R} \)^\gamma \)^{\beta-p}} dx \)^{\frac{p-2}{2}} &\text{if}\,\, p \in (1, 2)
\end{cases}
\end{align*}
for $\alpha = N-(\beta -1)\gamma$.
\end{theorem}


\begin{theorem}\label{T IH remainder 2}
Let $1 < p, \beta < \infty, \alpha \in \re, \gamma > 0$, and $\alpha \le N-(\beta -1)\gamma$. Then the inequality
\begin{align}\label{IH remainder 2}
&\( \frac{\beta -1}{p} \gamma \)^{p} \int_{B_R} \frac{|u|^{p}}{|x|^\alpha \( 1- \( \frac{|x|}{R} \)^\gamma \)^\beta } dx = \int_{B_R} \frac{\left| \nabla u \cdot \frac{x}{|x|} \right|^p}{|x|^{\alpha-p}\( 1-\( \frac{|x|}{R} \)^\gamma \)^{\beta-p}} dx \notag \\ 
&-p \( \frac{\beta -1}{p} \gamma \)^{p}  \int_{B_R} R_p \left[ u,\, \frac{p}{(\beta -1) \gamma} \( -\nabla u \cdot \frac{x}{|x|} \) \,|x| \, \( 1- \( \frac{|x|}{R} \)^\gamma \) \, \right] \frac{dx}{|x|^\alpha \( 1- \( \frac{|x|}{R} \)^\gamma \)^\beta }  \notag \\
&-(N-\alpha - (\beta -1) \gamma) \( \frac{\beta -1}{p} \gamma \)^{p-1} \int_{B_R} \frac{|u|^p}{|x|^{\alpha} \(1- \( \frac{|x|}{R} \)^\gamma \)^{\beta-1}} \,dx
\end{align}
holds for all $u \in C_c^1(B_R)$, where 
\begin{align*}
R_p (\xi, \eta ) &= \frac{1}{p} \,|\,\eta\,|^p + \frac{p-1}{p} \,|\,\xi\,|^p - \,|\,\xi\,|^{p-2} \xi \, \eta \\
&= (p-1) \int_0^1 |\,t \,\xi + (1-t) \,\eta\,|^{p-2} \,dt \, |\,\xi -\eta\,|^2 \ge 0.
\end{align*}
\end{theorem}

\begin{remark}\label{R virtual extremal}
Note that $R_p (\xi, \eta ) = 0$ if and only if $\xi = \eta$. Let $R=1$ for simplicity. 
If $R_p =0$ in (\ref{IH remainder 2}) for some function $u=u(x)= u(r\w)\,(r=|x|, \w \in \mathbb{S}^{N-1})$, then for any $r \in (0,1), \w \in \mathbb{S}^{N-1}$ we have
\begin{align*}
-\frac{\pd u}{\pd r} (r\w) = \frac{(\beta -1) \gamma}{p}  \frac{u(r\w)}{r \, \( 1- r^\gamma \)}.
\end{align*}
Here, for fixed $\w \in \mathbb{S}^{N-1}$, we set $g(r)= u(r\w)$. Then $g$ satisfies the following ODE:
\begin{align*}
-g'(r) = \frac{(\beta -1) \gamma}{p}  \frac{g(r)}{r \, \( 1- r^\gamma \)}, \quad g(1)=0
\end{align*}
We can solve it by separation of variables as follows:
\begin{align*}
g(r) = C_{\w}  (r^{-\gamma} -1)^{\frac{\beta -1}{p}} \quad (C_{\w} \in \re)
\end{align*}
This means that $u(x)= f \( \frac{x}{|x|} \) ( \,|x|^{-\gamma} -1)^{\frac{\beta -1}{p}}$ for some function $f: \mathbb{S}^{N-1} \to \re$. 
 
\end{remark}


\begin{proof}(Proof of Theorem \ref{T IH remainder}).
For the simplicity, we set $R=1$. Note that 
$$
{\rm div} \( \frac{x}{\( |x|^{-\gamma} -1 \)^{\beta-1}} \)
= \frac{N \,|x|^{(\beta -1)\gamma}}{\( 1-|x|^\gamma \)^{\beta -1}}
+ \frac{(\beta -1)\gamma \,|x|^{(\beta -1)\gamma}}{\( 1-|x|^\gamma \)^{\beta }}.
$$
Then we have 
\begin{align*}
&(\beta -1)\gamma \int_{B_1} \frac{|u|^p}{|x|^\alpha \( 1-|x|^\gamma \)^\beta} \,dx \\
&= \int_{B_1} {\rm div} \( \frac{x}{\( |x|^{-\gamma} -1 \)^{\beta-1}} \) \frac{|u|^p}{|x|^{\alpha + (\beta -1)\gamma}} - N \frac{|u|^p}{|x|^\alpha \( 1-|x|^\gamma \)^{\beta -1}} \,dx\\
&= -p \int_{B_1} \frac{|u|^{p-2}u \,(\nabla u \cdot x)}{|x|^{\alpha -1} \( 1-|x|^{\gamma} \)^{\beta-1}} \,dx  - (N-\alpha -(\beta -1)\gamma ) \frac{|u|^p}{|x|^\alpha \( 1-|x|^\gamma \)^{\beta -1}} \,dx\\
&\le p \( \int_{B_1} \frac{\left| \nabla u \cdot \frac{x}{|x|} \right|^p}{|x|^{\alpha -p} \( 1-|x|^\gamma \)^{\beta -p}} \,dx \)^{\frac{1}{p}} \( \int_{B_1} \frac{|u|^p}{|x|^\alpha \( 1-|x|^\gamma \)^\beta} \,dx  \)^{1-\frac{1}{p}} \\
&- (N-\alpha -(\beta -1)\gamma ) \int_{B_1} \frac{|u|^p}{|x|^\alpha \( 1-|x|^\gamma \)^{\beta -1}} \,dx
\end{align*}
which implies that for any $u \neq 0$
\begin{align}\label{p mae}
\frac{\beta-1}{p} \gamma \( \int_{B_1} \frac{|u|^p}{|x|^\alpha \( 1-|x|^\gamma \)^\beta} \,dx  \)^{\frac{1}{p}}
&\le  \( \int_{B_1} \frac{\left| \nabla u \cdot \frac{x}{|x|} \right|^p}{|x|^{\alpha -p} \( 1-|x|^\gamma \)^{\beta -p}} \,dx \)^{\frac{1}{p}}\notag \\
&-\frac{(N-\alpha -(\beta -1)\gamma )}{p} \frac{\int_{B_1} \frac{|u|^p}{|x|^\alpha \( 1-|x|^\gamma \)^{\beta -1}} \,dx}{\( \int_{B_1} \frac{|u|^p}{|x|^\alpha \( 1-|x|^\gamma \)^\beta} \,dx  \)^{\frac{p-1}{p}}}.
\end{align}
Therefore we obtain the inequality (\ref{IH gene}). 
Set 
\begin{align*}
A= \( \int_{B_1} \frac{\left| \nabla u \cdot \frac{x}{|x|} \right|^p}{|x|^{\alpha -p} \( 1-|x|^\gamma \)^{\beta -p}} \,dx \)^{\frac{1}{p}}, \quad
B= \frac{(N-\alpha -(\beta -1)\gamma )}{p} \frac{\int_{B_1} \frac{|u|^p}{|x|^\alpha \( 1-|x|^\gamma \)^{\beta -1}} \,dx}{\( \int_{B_1} \frac{|u|^p}{|x|^\alpha \( 1-|x|^\gamma \)^\beta} \,dx  \)^{\frac{p-1}{p}}}.
\end{align*}
By the fundamental inequality $(A-B)^p \le A^p -p (A-B)^{p-1} B\,(A \ge B)$ and the inequality
$A-B \ge \frac{\beta -1}{p} \gamma \( \int_{B_1} \frac{|u|^p}{|x|^\alpha \( 1-|x|^\gamma \)^\beta} \,dx  \)^{\frac{1}{p}}$ from (\ref{p mae}), we have
\begin{align*}
\(\frac{\beta-1}{p} \gamma \)^p \int_{B_1} \frac{|u|^p}{|x|^\alpha \( 1-|x|^\gamma \)^\beta} \,dx
&\le A^p -p(A-B)^{p-1} B\\
&\le \int_{B_1} \frac{\left| \nabla u \cdot \frac{x}{|x|} \right|^p}{|x|^{\alpha -p} \( 1-|x|^\gamma \)^{\beta -p}} \,dx - \psi_{N,p, \alpha, \beta} (u)
\end{align*}
which implies (\ref{IH remainder}) for $\alpha < N- (\beta -1)\gamma$. Assume that $\alpha = N- (\beta -1)\gamma$. For $u \in C^1_c (B_1)$, set
\begin{align*}
v(x) &= u(x) \( |x|^{-\gamma} -1 \)^{-\frac{\beta -1}{p} },\\
J(u) &= \int_{B_1} \frac{\left| \nabla u \cdot \frac{x}{|x|} \right|^p}{|x|^{\alpha-p}\( 1-|x|^\gamma \)^{\beta-p}} dx - \( \frac{\beta -1}{p} \gamma \)^{p} \int_{B_1} \frac{|u|^{p}}{|x|^\alpha \( 1- |x|^\gamma \)^\beta } dx.
\end{align*}
Then we have
\begin{align*}
\nabla u \cdot \frac{x}{|x|} = B-A,\,\text{where}\,\, A= \frac{\beta -1}{p}\gamma v \,|x|^{-\gamma -1} (|x|^{-\gamma} -1)^{\frac{\beta -1}{p} -1}, \,B= (|x|^{-\gamma} -1)^{\frac{\beta -1}{p}} \( \nabla v \cdot \frac{x}{|x|} \).
\end{align*}
By the inequality
\begin{align*}
|a-b|^p - |a|^p + p|a|^{p-2} ab \ge 
\begin{cases}
C_1 |b|^p \quad &\text{if} \,\, p \in [2,\infty),\\
C_2 \frac{|b|^2}{(|a-b| + |a|)^{2-p}} &\text{if} \,\, p \in (1,2)
\end{cases}
\end{align*}
for some $C_1, C_2>0$ and for any $a, b \in \re$ (See e.g. \cite{Lind}), we have
\begin{align*}
J(u) &\ge \int_{B_1} \frac{|A|^p -p |A|^{p-2}AB + C_1|B|^p}{|x|^{\alpha-p}\( 1-|x|^\gamma \)^{\beta-p}} dx - \( \frac{\beta -1}{p} \gamma \)^{p} \int_{B_1} \frac{|v|^{p}}{|x|^N \( 1- |x|^\gamma \) } dx\\
&= \( \frac{\beta -1}{p} \gamma \)^{p-1} \int_{B_1} \frac{\nabla \( |v|^p \) \cdot \frac{x}{|x|}}{|x|^{N-1}} \,dx + C_1 \int_{B_1} |x|^{p-N}\( 1- |x|^\gamma \)^{p-1} \left| \nabla v \cdot \frac{x}{|x|} \right|^p \,dx
\end{align*}
for $p \in [2, \infty)$. Since 
\begin{align*}
\int_{B_1} \frac{\nabla \( |v|^p \) \cdot \frac{x}{|x|}}{|x|^{N-1}} \,dx = \int_{\mathbb{S}^{N-1}} \int_0^1 \frac{\pd}{\pd r} (|v|^p) \,dr \,dS_\w =0,
\end{align*}
we have (\ref{IH remainder}) for $p \in [2, \infty)$ and $\alpha = N- (\beta -1)\gamma$. On the other hand, we have
\begin{align}\label{J est}
J(u) &\ge C_2 \int_{B_1} \frac{|B|^2}{|x|^{\alpha-p}\( 1-|x|^\gamma \)^{\beta-p} (|A-B| + |A|)^{2-p}} \,dx \notag\\
&=C_2 \int_{B_1} \frac{|x|^{p-\alpha -\frac{2}{p}(N-\alpha)} (1-|x|^\gamma)^{-\beta (1-\frac{2}{p}) +p -\frac{2}{p}} \left| \nabla v \cdot \frac{x}{|x|} \right|^2}
{\( \left| \nabla u \cdot \frac{x}{|x|} \right|
+
\frac{\beta-1}{p} \gamma |v| |x|^{-1-\frac{N-\alpha}{p}} \( 1-|x|^\gamma \)^{\frac{\beta-1}{p} -1}  \)^{2-p}} \,dx
\end{align}
for $p \in (1, 2)$. In the same way as the proof of Theorem 1 in \cite{II}, the H\"older inequality and (\ref{J est}) imply
\begin{align*}
&\int_{B_1} |x|^{p-N}\( 1- |x|^\gamma \)^{p-1} \left| \nabla v \cdot \frac{x}{|x|} \right|^p \,dx \\
&\le \( \int_{B_1} \frac{|x|^{p-\alpha -\frac{2}{p}(N-\alpha)} (1-|x|^\gamma)^{-\beta (1-\frac{2}{p}) +p -\frac{2}{p}} \left| \nabla v \cdot \frac{x}{|x|} \right|^2}
{\( \left| \nabla u \cdot \frac{x}{|x|} \right|
+
\frac{\beta-1}{p} \gamma |v| |x|^{-1-\frac{N-\alpha}{p}} \( 1-|x|^\gamma \)^{\frac{\beta-1}{p} -1}  \)^{2-p}} \,dx \)^{\frac{p}{2}}\\
&\times \( \int_{B_1} \left| \, \left| \nabla u \cdot \frac{x}{|x|} \right|
+
\frac{\beta-1}{p} \gamma \frac{|v|}{|x|^{1+\frac{N-\alpha}{p}}} \( 1-|x|^\gamma \)^{\frac{\beta-1}{p} -1} \right|^p |x|^{p-\alpha} (1-|x|^\gamma)^{p-\beta}\,dx \)^{\frac{2-p}{2}} \\
&\le \( \frac{J(u)}{C_2} \)^{\frac{p}{2}} 2^{\frac{(p-1)(2-p)}{2}} \( \int_{B_1} \frac{\left| \nabla u \cdot \frac{x}{|x|} \right|^p}{|x|^{\alpha-p}\( 1-|x|^\gamma \)^{\beta-p}} dx \)^{\frac{2-p}{2}}.
\end{align*}
Therefore we have (\ref{IH remainder}) for $p \in (1, 2)$ and $\alpha = N- (\beta -1)\gamma$. \qed
\end{proof}


\begin{proof}(Proof of Theorem \ref{T IH remainder 2}).
We refer \cite{IIO(2016),IIO(2017)}. 
For the simplicity, we set $R=1$. Note that 
\begin{align*}
\frac{d}{dr} \left[ \( r^{-\gamma} -1 \)^{-\beta+1} \right]
= (\beta -1)\gamma \,r^{-1 + (\beta -1)\gamma} \( 1- r^{\gamma} \)^{-\beta}
\end{align*}
Then we have 
\begin{align*}
&\int_{B_1} \frac{|u|^p}{|x|^\alpha \( 1-|x|^\gamma \)^\beta} \,dx \\
&= \int_0^1 r^{N-\alpha -1} (1-r^\gamma )^{-\beta} \int_{\mathbb{S}^{N-1}} |u(r\w )|^p \, dr dS_{\w} \\
&= \frac{1}{(\beta -1) \gamma} \int_0^1 \int_{\mathbb{S}^{N-1}} |u|^p r^{N-\alpha -(\beta -1)\gamma} \,\frac{d}{dr} \left[ \( r^{-\gamma} -1 \)^{-\beta+1} \right] \, dr dS_{\w} \\
&=\frac{p}{(\beta -1)\gamma} \iint |u|^{p-2}u \( - \frac{\pd u}{\pd r} \) r^{N-\alpha} (1-r^\gamma )^{-\beta +1} \, dr dS_{\w}\\
&- \frac{N-\alpha -(\beta -1) \gamma}{(\beta -1)\gamma} \iint |u|^{p} r^{N-\alpha-1} (1-r^\gamma )^{-\beta +1} \, dr dS_{\w}\\
&=\iint |\,\xi \,|^{p-2} \xi \eta \,drdS_{\w} - \frac{N-\alpha -(\beta -1) \gamma}{(\beta -1)\gamma}  \int_{B_1} \frac{|u|^p}{|x|^{\alpha} \(1- |x|^\gamma \)^{\beta-1}} \,dx,
\end{align*}
where 
\begin{align*}
\eta = \frac{p}{(\beta -1)\gamma} \( -\frac{\pd u}{\pd r} (r\w  ) \) r^{\frac{N-\alpha +p -1}{p}} (1- r^\gamma )^{\frac{-\beta +p}{p}}, \quad \xi = u(r\w) \,r^{\frac{N-\alpha -1}{p}} (1- r^\gamma )^{\frac{-\beta}{p}}.
\end{align*}
Since $|\,\xi \,|^{p-2} \xi \eta = \frac{1}{p} \,|\,\eta\,|^p + \frac{p-1}{p} \,|\,\xi\,|^p - R_p(\xi, \,\eta)$, we have
\begin{align*}
\int_{B_1} \frac{|u|^p}{|x|^\alpha \( 1-|x|^\gamma \)^\beta} \,dx 
&=\frac{1}{p} \( \frac{p}{(\beta -1)\gamma} \)^p \iint \left| \frac{\pd u}{\pd r} \right|^p r^{N-\alpha +p -1} (1-r^\gamma )^{-\beta +p} \,drdS_{\w} \\
&+\frac{p-1}{p} \iint |u|^p r^{N-\alpha -1} (1-r^\gamma )^{-\beta} \,drdS_{\w} -\iint R_p (\xi, \,\eta) \,drdS_\w \\
&- \frac{N-\alpha -(\beta -1) \gamma}{(\beta -1)\gamma}  \int_{B_1} \frac{|u|^p}{|x|^{\alpha} \(1- |x|^\gamma \)^{\beta-1}} \,dx
\end{align*}
which implies (\ref{IH remainder 2}). 
\qed
\end{proof}


Here, we show Theorem \ref{T IH} by using Theorem \ref{T IH remainder}.


\begin{proof}(Proof of Theorem \ref{T IH}).
Let $\alpha \le N - (\beta -1)\gamma$. For the simplicity, we set $R=1$. 
We show the optimality of the constant $\( \frac{\beta-1}{p} \gamma \)^p$ in (\ref{IH gene}). For $A > \frac{\beta-1}{p}$ and small $\delta >0$, set
\begin{align*}
	f_A (x)= \phi_\delta (x) \,(1-|x|^\gamma)^A,
\end{align*}
where $\phi_\delta$ is a smooth radially symmetric function which satisfies $0 \le \phi_\delta \le 1, \phi_\delta \equiv 0$ on $B_{1-2\delta}$ and $\phi_\delta \equiv 1$ on $B_1 \setminus B_{1-\delta}$.
Then we have
\begin{align*}
\( \frac{\beta-1}{p} \gamma \)^p 
&\le \frac{\int_{B_1} \frac{\left| \nabla f_A \cdot \frac{x}{|x|} \right|^p}{|x|^{\alpha -p} \( 1-|x|^\gamma \)^{\beta -p}} \,dx}{\int_{B_1} \frac{|f_A|^p}{|x|^\alpha \( 1-|x|^\gamma \)^\beta} \,dx} \\
&\le \frac{(A\gamma)^p \int_{1-\delta}^1 \( 1-r^\gamma \)^{Ap-\beta} r^{N-1-\alpha + \gamma p}\,dr + \int_{2-\delta}^{1-\delta} |(f_A)'|^p \( 1-r^\gamma \)^{p-\beta} r^{N-1-\alpha+p}\,dr}{\int_{1-\delta}^{1} \( 1-r^\gamma \)^{Ap-\beta} r^{N-1-\alpha}\,dr} \\
&= \( \frac{\beta-1}{p}\gamma \)^p + o(1) \quad \( A \to \frac{\beta-1}{p} \).
\end{align*}
Therefore the constant $\( \frac{\beta-1}{p} \gamma \)^p$ in (\ref{IH gene}) is optimal. Since there exists the nonnegative remainder term in Theorem \ref{T IH remainder} when $\alpha \le N - (\beta -1)\gamma$, we observe that if there exists an extremal function $U=U(x)$ of the inequality (\ref{IH gene}), then $U(x)= c \( |x|^{-\gamma} -1 \)^{\frac{\beta -1}{p}} = c |x|^{-\frac{\beta -1}{p}\gamma} \(  1-|x|^\gamma \)^{\frac{\beta -1}{p}}$ for some $c \in \re$. However, if $c \not= 0$, then the right-hand side of (\ref{IH gene}) diverges since
\begin{align*}
\int_{B_1 \setminus B_{1-\ep}} \frac{\left| \nabla U \cdot \frac{x}{|x|} \right|^p}{|x|^{\alpha -p} \( 1-|x|^\gamma \)^{\beta -p}} \,dx
&\ge C(\ep) \int_{B_1 \setminus B_{1-\ep}} \frac{\left| \nabla \(  1-|x|^\gamma \)^{\frac{\beta -1}{p}} \right|^p}{ \( 1-|x|^\gamma \)^{\beta -p}} \,dx + D(\ep)\\
&\ge \tilde{C}(\ep) \int_{B_1 \setminus B_{1-\ep}} \( 1-|x|^\gamma \)^{-1} \,dx + D(\ep) =\infty
\end{align*}
for any small $\ep >0$, where $C(\ep), \tilde{C}(\ep), D(\ep)$ are some constants depending on $\ep$.
The proof of Theorem \ref{T IH} is now complete.
\qed
\end{proof}

\subsection{Transformation approach and an improved Hardy-Sobolev type inequality}\label{Sec trans}

In this subsection, we show the generalized inequalities (\ref{IH gene}),  (\ref{IH log}) via the following transformation which is a generalization of harmonic transplantation proposed by Hersch \cite{Hersch}, see also \cite{F book,BBF}. Concerning a summary of harmonic transplantation, see \S 3 in \cite{ST(HT)}. Consider
\begin{align}\label{trans}
u(x) = v(y)= w (z),\,\, \text{where}\,\, \( |x|^{-\gamma} -R^{-\gamma} \) \frac{x}{|x|} = |y|^{-\gamma}\frac{y}{|y|} = \( \log \frac{R}{|z|} \) \frac{z}{|z|}
\end{align}
and set $\gamma = \frac{N-\alpha}{\beta -1}$. Then we see that
\begin{align*}
&\int_{B_R} \frac{\left| \nabla u \cdot \frac{x}{|x|} \right|^p}{|x|^{\alpha-p}\( 1-\(\frac{|x|}{R} \)^\gamma \)^{\beta-p}} dx
=\int_{\re^N} \frac{\left| \nabla v \cdot \frac{y}{|y|} \right|^p}{|y|^{\alpha-p}} dy
=\gamma^{p-1} \int_{B_R} \frac{\left| \nabla w \cdot \frac{z}{|z|} \right|^p}{|z|^{N-p}\( \log \frac{R}{|z|}  \)^{\beta-p}} dz, \\
&\int_{B_R} \frac{| u|^p}{|x|^{\alpha}\( 1-\(\frac{|x|}{R} \)^\gamma \)^{\beta}} dx
=\int_{\re^N} \frac{| v|^p}{|y|^{\alpha}} dy
=\gamma^{-1} \int_{B_R} \frac{| w|^p}{|z|^{N} \( \log \frac{R}{|z|}  \)^{\beta}} dz.
\end{align*}
Therefore the inequality (\ref{H_p}) on $\re^N$ for $v$ is equivalent to the inequality (\ref{IH gene}) with $\alpha = N -(\beta -1)\gamma$ for $u$ and the inequality (\ref{IH log}) with $\alpha = N$ for $w$. 
Moreover, since the inequality (\ref{H_p}) on $\re^N$ for $v$ is invariant under the usual scaling $v_\la (y)= \la^{\frac{N-\alpha}{p}} v(\tilde{y})$, where $\tilde{y}= \la y$ and $\la >0$, we obtain scale invariance structures of (\ref{IH gene}) and (\ref{IH log}) thanks to the transformations (\ref{trans}) as follows. 

\begin{proposition}\label{P scale}
The inequality (\ref{IH gene}) with $\alpha = N-(\beta -1)\gamma$ for $u$ is invariant under the scaling $u_\la (x)= \la^{\frac{N-\alpha}{p}} u(\tilde{x})$, where $\tilde{x}= \la x \left[ 1- (1-\la^\gamma) \( \frac{|x|}{R} \)^\gamma \right]^{-\frac{1}{\gamma}}$ and $\la >0$. 
On the other hand, the inequality (\ref{IH log}) with $\alpha = N$ for $w$ is invariant under the scaling $w_\mu (z)= \mu^{-\frac{\beta -1}{p}} w(\tilde{z})$, where $\tilde{z}= \( \frac{|z|}{R} \)^{\mu -1} z$ and $\mu = \la^{-\gamma} >0$. 
\end{proposition}

\begin{remark}\label{R scale}
If $\alpha < N-(\beta -1)\gamma$, then we see that the inequality (\ref{IH gene}) is not invariant under the scaling $u_\la (r\w)= \la^A u(s \w)$, where $s=s(r), |x|=r, |\tilde{x}|=s, \frac{x}{|x|} = \frac{\tilde{x}}{|\tilde{x}|} =\w$ and $s'(r) >0, s(0)=0, s(1)=1$. 
In fact, assume that the inequality (\ref{IH gene}) is invariant as follows.
\begin{align}\label{scale 1}
\int_{B_R} \frac{\left| \nabla u_\la \cdot \frac{x}{|x|} \right|^p}{|x|^{\alpha-p}\( 1-\(\frac{|x|}{R} \)^\gamma \)^{\beta-p}} dx 
&= \int_{B_R} \frac{\left| \nabla u \cdot \frac{\tilde{x}}{|\tilde{x}|} \right|^p}{|\tilde{x}|^{\alpha-p}\( 1-\(\frac{|\tilde{x}|}{R} \)^\gamma \)^{\beta-p}} d\tilde{x}, \\
\label{scale 2}
\int_{B_R} \frac{| u_\la|^p}{|x|^{\alpha}\( 1-\(\frac{|x|}{R} \)^\gamma \)^{\beta}} dx
&= \int_{B_R} \frac{| u|^p}{|\tilde{x}|^{\alpha}\( 1-\(\frac{|\tilde{x}|}{R} \)^\gamma \)^{\beta}} d\tilde{x}.
\end{align}
For the simplicity, we set $R=1$. Since
\begin{align*}
\int_{B_1} \frac{\left| \nabla u \cdot \frac{\tilde{x}}{|\tilde{x}|} \right|^p}{|\tilde{x}|^{\alpha-p}\( 1-|\tilde{x}|^\gamma \)^{\beta-p}} d\tilde{x}
&= \int_{\mathbb{S}^{N-1}} \int_0^1 \left| \frac{\pd u}{\pd s} \right|^p s^{N-1-\alpha +p} (1-s^\gamma)^{-\beta +p} \,ds \,dS_\w,\\
\int_{B_1} \frac{\left| \nabla u_\la \cdot \frac{x}{|x|} \right|^p}{|x|^{\alpha-p}\( 1-|x|^\gamma \)^{\beta-p}} dx 
&=\int_{\mathbb{S}^{N-1}} \int_0^1 \left| \frac{\pd u}{\pd s} \right|^p \la^{Ap} \(  \frac{ds}{dr} \)^{p-1} r^{N-1-\alpha +p} (1-r^\gamma)^{-\beta +p} \,ds \,dS_\w,
\end{align*}
(\ref{scale 1}) implies that 
\begin{align}\label{relation 1}
\( \frac{ds}{dr} \)^{p-1} \cdot \frac{\la^{Ap} r^{N-1-\alpha} (1-r^\gamma)^{-\beta} }{s^{N-1-\alpha} (1-s^\gamma)^{-\beta}} = \( \frac{r (1-r^\gamma) }{s (1-s^\gamma)} \)^p.
\end{align}
In the same way as above, we have 
\begin{align}\label{relation 2}
\frac{ds}{dr}  = \la^{Ap}  \frac{r^{N-1-\alpha} (1-r^\gamma)^{-\beta} }{s^{N-1-\alpha} (1-s^\gamma)^{-\beta}} 
\end{align}
from (\ref{scale 2}). By (\ref{relation 1}) and (\ref{relation 2}), we have
\begin{align*}
s^{N-\alpha} (1-s^\gamma)^{1-\beta} = \la^{Ap} r^{N-\alpha} (1-r^\gamma)^{1-\beta}
\end{align*}
which implies that
\begin{align*}
\frac{ds}{dr}  = \la^{Ap}  \frac{r^{N-\alpha-1} (1-r^\gamma)^{-\beta}}{s^{N-\alpha -1} (1-s^\gamma)^{-\beta}} \cdot \frac{\left[ 1- \left\{ \frac{N-\alpha -(\beta-1)\gamma}{N-\alpha} \right\} r^\gamma \right]}{\left[ 1- \left\{ \frac{N-\alpha -(\beta-1)\gamma}{N-\alpha} \right\} s^\gamma \right]}.
\end{align*}
Therefore, we have $\alpha = N- (\beta -1)\gamma$ by comparing it with (\ref{relation 2}). 

In the same way as above, we can also show that the inequality (\ref{IH log}) with $\alpha < N$ is not invariant under the scaling $u_\la (r\w)= \la^A u(s \w)$. 
\end{remark}


The following minimization problem associated with the Hardy-Sobolev type inequality is well-known.

\begin{ThmA}(Ref. \cite{H} Lemma 3.1 or \cite{HK})
Let $1<p<+\infty, N \ge 2$, and $W^{1,p}_{A, B, {\rm rad}}(\re^N)$ be the completion of $C_{c, {\rm rad}}^{\infty}(\re^N)$ with respect to the norm $\| \nabla (\cdot )\|_{L^p(\re^N ; \, |y|^{Ap} dy)}$. Assume that $p, q, N, A,$ and $B$ satisfy
\begin{align}\label{condi}
(1-A+B)p <N, 0 < \frac{1}{p} -\frac{1}{q} = \frac{1-A+B}{N}, -\frac{N}{q} < B.
\end{align} 
Under these assumptions we set
\begin{align*}
S_{{\rm rad}} = \inf \left\{ \, \int_{\re^N} | \nabla v |^p |y|^{Ap}\,dx  \,\,\Biggr| \,\, v \in W^{1,p}_{A, B, {\rm rad}}(\re^N),\, \int_{\re^N} | v |^q |y|^{Bq} dx =1 \, \right\}.
\end{align*}
Then 
\begin{align*}
S_{{\rm rad}} &= \pi^{\frac{p(1-A+B)}{2}} N \( \frac{N-(1-A+B)p}{p-1} \)^{p-1} \( \frac{N-p+pA}{N-(1-A+B)p} \)^{p-\frac{p(1-A+B)}{N}}
\times \\
&\( \frac{2(p-1)}{(1-A+B)p} \)^{\frac{p(1-A+B)}{N}} \left\{ \, \frac{\Gamma \( \frac{N}{p(1-A+B)} \) \,\Gamma \( \frac{N(p-1)}{p(1-A+B)} \)}{\Gamma \( \frac{N}{2} \) \, \Gamma \( \frac{N}{1-A+B} \)} \, \right\}^{\frac{p(1-A+B)}{N}}.
\end{align*}
Moreover, $S_{{\rm rad}}$ is attained by functions of the form
\begin{align*}
V(y)=\left[ a + b |y|^{\frac{ph}{p-1}} \right]^{1-\frac{N}{p(1-A+B)}}\quad \( \,a, b >0, h= \frac{(1-A+B)(N-p+pA)}{N-(1-A+B)p} \,\)
\end{align*}
\end{ThmA}

Note that the assumption (\ref{condi}) is equivalent to the following:
\begin{align}\label{condi 2}
p < q = q(p)= \frac{Np}{N- (1-A+ B) p} < +\infty, \quad A > \frac{p-N}{p}
\end{align}
By Theorem A and the transformation (\ref{trans}), we can also obtain several results for the minimization problem $T_{{\rm rad}}$ associated with an improved Hardy-Sobolev type inequality for radially symmetric functions. In fact, we set
\begin{align*}
T_{{\rm rad}} := \inf \left\{ \, \int_{B_R} \frac{| \nabla u |^p}{|x|^{\alpha-p}\( 1-\(\frac{|x|}{R} \)^\gamma \)^{\beta-p}} dx  \,\,\Biggr| \,\, u \in X^{1,p}_{\alpha, \beta, {\rm rad}},\, \int_{B_R} \frac{| u |^q}{|x|^{N- \frac{q}{p} (N-\alpha)}\( 1-\(\frac{|x|}{R} \)^\gamma \)^{1+\frac{(\beta -1)}{p} q}} dx =1 \, \right\},
\end{align*}
where $q>p, \alpha < N, \beta >1$, $\gamma = \frac{N-\alpha}{\beta -1}$, and $X^{1,p}_{A, B, {\rm rad}}$ be the completion of $C_{c, {\rm rad}}^{\infty}(B_R)$ with respect to the norm $\| \nabla (\cdot )\|_{L^p(B_R ; \, |x|^{p-\alpha} (1- (|x| /R)^\gamma )^{p-\beta} dx)}$. 
By using the transformation (\ref{trans}), we have
\begin{align*}
&\int_{\re^N} |\nabla v|^p |y|^{pA} dy
=\int_{B_R} \frac{|\nabla u|^p |x|^{pA}}{\( 1- \( \frac{|x|}{R} \)^\gamma  \)^{\beta-p}} dx
=\int_{B_R} \frac{|\nabla u|^p}{|x|^{\alpha -p} \( 1- \( \frac{|x|}{R} \)^\gamma  \)^{\beta-p}} dx, \\
&\int_{\re^N} | v|^q |y|^{Bq} dy
=\int_{B_R} \frac{| u|^q}{|x|^{-Bq}\( 1-\(\frac{|x|}{R} \)^\gamma \)^{1+ \frac{N+ Bq}{\gamma}}} dx
= \int_{B_R} \frac{| u |^q}{|x|^{N- \frac{q}{p} (N-\alpha)}\( 1-\(\frac{|x|}{R} \)^\gamma \)^{1+\frac{(\beta -1)}{p} q}} dx,
\end{align*}
where $\alpha = p-pA, q= \frac{Np}{N-p + Ap -Bp} = \frac{Np}{(\beta -1)\gamma -Bp}, -Bq = N- \frac{q}{p} (\beta -1)\gamma = N- \frac{q}{p} (N-\alpha)$ from (\ref{condi 2}). 
Therefore, we obtain the following from Theorem A.

\begin{theorem}\label{T IHS rad}
Let $1<p<q<+\infty, N \ge 2, \alpha < N, \beta >1,$ and $\gamma = \frac{N-\alpha}{\beta -1}$. 
Then 
\begin{align*}
T_{{\rm rad}} = \pi^{\frac{N(q-p)}{2q}} N \(\frac{Np}{q(p-1)} \)^{p-1} \( \frac{q(N-\alpha )}{Np} \)^{p-1+\frac{p}{q}}
 \( \frac{2q(p-1)}{N(q-p)} \)^{\frac{q-p}{q}} \left\{ \, \frac{\Gamma \( \frac{q}{q-p} \) \,\Gamma \( \frac{q(p-1)}{q-p} \)}{\Gamma \( \frac{N}{2} \) \, \Gamma \( \frac{qp}{q-p} \)} \, \right\}^{\frac{q-p}{q}}.
\end{align*}
Moreover, $T_{{\rm rad}}$ is attained by functions of the form
\begin{align*}
U(x)=\left[ a + b \( \,\( \frac{|x|}{R} \)^{-\gamma} -1 \)^{-\frac{(q-p)(\beta -1)}{p(p-1)}} \right]^{-\frac{p}{q-p}}\quad \( \,a, b >0 \,\)
\end{align*}
\end{theorem}

\begin{remark}
In Theorem \ref{T IHS rad}, the result in the case where $\beta = p$ coincides with Theorem 1.3 in \cite{I}. In the special case where $\alpha = \beta =p$, a modified minimization problem is also studied without radially symmetry and the transformation, see \cite{S(NA)}. 
The optimal constant and the non-attainability in the case where $q=p$ are studied in Theorem \ref{T IH} in the present paper. 
A minimization problem with logarithmic weight is studied by \cite{HK,II(IMRN),ST,S(JDE)}. 
\end{remark}


\subsection{The divergence theorem and generalizations of two Hardy type inequalities: Proof of Theorem \ref{T IH another}}\label{Sec 2.3}

We show Theorem \ref{T IH another} by using the divergence theorem in a slightly different way from the proof of Theorem \ref{T IH}. 

\begin{proof}(Proof of Theorem \ref{T IH another}) 
For the simplicity, we set $R=1$. Note that 
\begin{align}\label{Holder mae}
\int_{B_1} \frac{(N-\alpha ) \,|u|^p}{|x|^\alpha \( 1- |x|^\gamma \)^\beta}
+ \frac{\beta \,\gamma \,|u|^p}{|x|^{\alpha - \gamma} \( 1- |x|^\gamma \)^{\beta+1}} \,dx = \int_{B_1} {\rm div} \( \frac{x}{|x|^\alpha \( 1- |x|^\gamma \)^\beta} \)\,|u|^p \,dx.
\end{align}

\noindent
(I) If $\alpha =N$, the inequality (\ref{IH gene another}) is trivial. Therefore, we assume that $\alpha < N$. If we drop the second term on the left-hand side of (\ref{Holder mae}), then we have
\begin{align*}
\( \frac{N-\alpha}{p} \) \int_{B_1} \frac{|u|^p}{|x|^\alpha \( 1- |x|^\gamma \)^\beta} \,dx
&\le \int_{B_1} \frac{|u|^{p-1} \, \left| \nabla u \cdot \frac{x}{|x|} \right|}{|x|^{\alpha -1} \( 1- |x|^\gamma \)^\beta} \,dx \\
&\le  \( \int_{B_1} \frac{|u|^p}{|x|^\alpha \( 1- |x|^\gamma \)^\beta} \,dx \)^{1-\frac{1}{p}}
\( \int_{B_1} \frac{\left| \nabla u \cdot \frac{x}{|x|} \right|^p}{|x|^{\alpha -p} \( 1- |x|^\gamma \)^\beta} \,dx \)^{\frac{1}{p}}
\end{align*}
which implies the desired inequality (\ref{IH gene another}) for functions $u \in C_c^1 (B_1)$. In order to show the optimality of the constant $( \frac{N-\alpha}{p} )^p$ in (\ref{IH gene another}), we consider the test function $f_A (x)= \phi_\delta (x) \,|x|^A$ 
for $A < \frac{N-\alpha}{p}$ and small $\delta >0$, 
where $\phi_\delta$ is a smooth radially symmetric function which satisfies $0 \le \phi_\delta \le 1, \phi_\delta \equiv 0$ on $B_1 \setminus B_{2\delta}$ and $\phi_\delta \equiv 1$ on $B_{\delta}$. Then we have
\begin{align*}
\( \frac{N-\alpha}{p} \)^p 
&\le \frac{\int_{B_1} \frac{\left| \nabla f_A \cdot \frac{x}{|x|} \right|^p}{|x|^{\alpha -p} \( 1-|x|^\gamma \)^{\beta}} \,dx}{\int_{B_1} \frac{|f_A|^p}{|x|^\alpha \( 1-|x|^\gamma \)^\beta} \,dx} \\
&\le \frac{A^p \int_0^{\delta} \( 1-r^\gamma \)^{-\beta} r^{Ap + N-\alpha -1}\,dr + \int_{\delta}^{2\delta} |(\,\phi_\delta \,r^A)'|^p \( 1-r^\gamma \)^{-\beta} r^{N-1-\alpha+p}\,dr}{\int_0^{\delta} \( 1-r^\gamma \)^{-\beta} r^{Ap +N-\alpha -1}\,dr} \\
&= \( \frac{N-\alpha}{p} \)^p + o(1) \quad \( A \to \frac{N-\alpha}{p} \).
\end{align*}
Therefore the constant $\( \frac{N-\alpha}{p} \)^p$ in (\ref{IH gene another}) is optimal. 
If $\beta >0$, then we can observe that the optimal constant $\( \frac{N-\alpha}{p} \)^p$ is not attained since we drop the second term on the left-hand side of (\ref{Holder mae}) when we show the inequality (\ref{IH gene another}). If $\beta =0$, then the inequality (\ref{IH gene another}) becomes the classical Hardy type inequality in which the non-attainability is well-known. Thus, we omit the proof.

\noindent
(II) If we drop the first term on the left-hand side of (\ref{Holder mae}), then we have
\begin{align*}
&\( \frac{\beta -1}{p} \gamma \) \int_{B_1} \frac{|u|^p}{|x|^{\alpha - \gamma} \( 1- |x|^\gamma \)^{\beta }} \,dx
\le \int_{B_1} \frac{|u|^{p-1} \, \left| \nabla u \cdot \frac{x}{|x|} \right|}{|x|^{\alpha -1} \( 1- |x|^\gamma \)^{\beta -1}} \,dx \\
&\le \( \int_{B_1} \frac{|u|^p}{|x|^{\alpha - \gamma} \( 1- |x|^\gamma \)^{\beta }} \,dx \)^{1-\frac{1}{p}} \( \int_{B_1} \frac{\left| \nabla u \cdot \frac{x}{|x|} \right|^p}{|x|^{\alpha - \gamma +(\gamma -1) p} \( 1- |x|^\gamma \)^{\beta -p}} \,dx \)^{\frac{1}{p}}
\end{align*}
which implies the desired inequality (\ref{IH gene another 2}) for functions $u \in C_c^1 (B_1 \setminus \{ 0\})$. 
We shall show the inequality (\ref{IH gene another 2}) for functions $u \in C_c^1 (B_1)$. 
For $u \in C_c^1(B_1)$, we consider $u_\ep = u \, (1- \varphi_\ep ) \in C_c^1 (B_1 \setminus \{ 0\})$, where $\varphi_\ep \in C_{c, {\rm rad}}^\infty (B_1), 0 \le \varphi_\ep \le 1, \varphi_\ep \equiv 1$ on $B_\ep, \,\varphi_\ep \equiv 0$ on $B_1 \setminus B_{2\ep}$, and $|\nabla \varphi_\ep | \le C \ep^{-1}$. 
Then we can derive the inequality (\ref{IH gene another 2}) for $u \in C_c^1(B_1)$ from the the inequality (\ref{IH gene another 2}) for $u_\ep \in C_c^1(B_1)$ and the assumption of $\alpha$ by taking the limit of it as $\ep \to 0$ since  
\begin{align*}
&\int_{B_1} \frac{\left| \nabla (u - u_\ep ) \cdot \frac{x}{|x|} \right|^p}{|x|^{\alpha - \gamma +(\gamma -1) p} \( 1- |x|^\gamma \)^{\beta -p}} \,dx \\
&\le \int_{B_\ep} \frac{\left| \nabla u  \cdot \frac{x}{|x|} \right|^p}{|x|^{\alpha - \gamma +(\gamma -1) p} \( 1- |x|^\gamma \)^{\beta -p}} \,dx  
+ \int_{B_{2\ep} \setminus B_\ep } \frac{\left| \nabla (u \varphi_\ep ) \cdot \frac{x}{|x|} \right|^p}{|x|^{\alpha - \gamma +(\gamma -1) p} \( 1- |x|^\gamma \)^{\beta -p}} \,dx \\
&\le C \| \nabla u \|_{\infty}^p \,\int_0^{2\ep} r^{N-1 -\alpha+ \gamma - (\gamma -1)p} \, dr + C \,\| u\|_{\infty}^p \ep^{-p}  \int_{0}^{2\ep} r^{N-1-\alpha +\gamma -  (\gamma -1) p}\, dr  \to 0 \quad (\ep \to 0),\\
&\int_{B_1} \frac{|u-u_\ep |^p}{|x|^{\alpha - \gamma} \( 1- |x|^\gamma \)^{\beta }} \,dx \to 0 \quad (\ep \to 0).
\end{align*}
The optimality of the constant $( \frac{\beta -1}{p} \gamma )^p$ in (\ref{IH gene another 2}) can be shown by the same test function $f_{A}$ in the proof of Theorem \ref{T IH}. Therefore, we omit the proof. If $\alpha < N$, then we can observe that the optimal constant $\( \frac{\beta -1}{p} \gamma \)^p$ is not attained since we drop the first term on the left-hand side of (\ref{Holder mae}) when we show the inequality (\ref{IH gene another 2}). Therefore, we shall show the non-attainability of the optimal constant $\( \frac{\beta -1}{p} \gamma \)^p$ in the case where $\alpha =N$. 
In the same way as the proof of Theorem \ref{T IH remainder 2}, we have
\begin{align*}
\int_{B_1} \frac{|u|^p}{|x|^{N-\gamma} \( 1-|x|^\gamma \)^\beta} \,dx 
&= \frac{1}{(\beta -1) \gamma} \int_0^1 \int_{\mathbb{S}^{N-1}} |u|^p \,\frac{d}{dr} \left[ \( 1- r^{\gamma} \)^{-\beta+1} \right] \, dr dS_{\w} \\
&=\frac{p}{(\beta -1)\gamma} \iint |u|^{p-2}u \( - \frac{\pd u}{\pd r} \) (1-r^\gamma )^{-\beta +1} \, dr dS_{\w} -|u(0)|^p \w_{N-1} \\
&\le \iint |\,\xi \,|^{p-2} \xi \eta \,drdS_{\w},
\end{align*}
where 
\begin{align*}
\eta = \frac{p}{(\beta -1)\gamma} \( -\frac{\pd u}{\pd r} (r\w  ) \) r^{-\frac{(\gamma -1)(p -1)}{p}} (1- r^\gamma )^{\frac{-\beta +p}{p}}, \quad 
\xi = u(r\w) \,r^{\frac{\gamma -1}{p}} (1- r^\gamma )^{\frac{-\beta}{p}}.
\end{align*}
Since $|\,\xi \,|^{p-2} \xi \eta = \frac{1}{p} \,|\,\eta\,|^p + \frac{p-1}{p} \,|\,\xi\,|^p - R_p(\xi, \,\eta)$, we have
\begin{align*}
\( \frac{\beta -1}{p} \gamma \)^p \int_{B_1} \frac{|u|^p}{|x|^{N-\gamma} \( 1-|x|^\gamma \)^\beta} \,dx 
&+ p \( \frac{\beta -1}{p} \gamma \)^p \iint R_p(\xi,\, \eta) \, drdS_\w \\
&\le \int_{B_1} \frac{\left| \nabla u \cdot \frac{x}{|x|} \right|^p}{|x|^{N-\gamma + (\gamma -1)p} \( 1-|x|^\gamma \)^{\beta - p}} \,dx.
\end{align*}
If we assume that the optimal constant $\( \frac{\beta -1}{p} \gamma \)^p$ is attained by some function $u$, then $R_p(\xi, \,\eta) =0$ which implies that $\xi = \eta$. Therefore, for fixed $\w \in \s$, $g(r) = u(r\w)$ satisfies following ODE:
\begin{align*}
-g'(r) = \( \frac{\beta -1}{p} \gamma \) \frac{r^{\gamma -1}}{1-r^\gamma} \, g(r), \,\, r \in (0,1), \quad g(1) = 0 
\end{align*}
Since $g(r) = (1-r^\gamma )^{\frac{\beta -1}{p}}$, $u(x) = (1-|x|^\gamma )^{\frac{\beta -1}{p}} f\( \frac{x}{|x|} \)$ for some function $f: \s \to \re$. However, the right-hand side of the inequality (\ref{IH gene another 2}) diverges since 
\begin{align*}
\int_{B_1} \frac{\left| \nabla u \cdot \frac{x}{|x|} \right|^p}{|x|^{N-\gamma + (\gamma -1)p} \( 1-|x|^\gamma \)^{\beta - p}} \,dx
&=  \( \frac{\beta -1}{p} \gamma \)^p \int_{\s} f(\w) \int_0^1 \frac{r^{\gamma -1}}{1-r^\gamma} \,drdS_\w \\
&= \( \frac{\beta -1}{p} \gamma \)^p \int_{\s} f(\w) \left[  -\log (1-r^\gamma ) \right]^1_0 \,dS_\w = +\infty.
\end{align*}
Therefore, the optimal constant $\( \frac{\beta -1}{p} \gamma \)^p$ is not attained even in the case where $\alpha =N$. 

\qed
\end{proof}



\section{Higher order inequalities}\label{Sec Higher}

The higher order generalization of Hardy type inequalities (\ref{H_p}), (\ref{H_p geo}) are called Rellich type inequalities due to the celebrated work by Rellich \cite{Rellich}. 
In this section, we consider the higher order generalization of Theorem \ref{T IH}.

Let $k, m \in \N, k \ge 2, p >1$, and 
\begin{align*}
&\nabla^k u = \begin{cases}
              \lap^m u  \quad &\text{if} \,\, k=2m,\\
              \nabla \lap^m u  &\text{if} \,\, k=2m+1,
              \end{cases}\\
&A_{k,p,\alpha} = \begin{cases}\vspace{0.5em}
              \prod_{j =0}^{m-1} \dfrac{\{ N-\alpha + 2j p \} \{ N(p-1) +\alpha -2(j +1)p \} }{p^2} \quad &\text{if} \,\, k=2m,\\
              \frac{N-\alpha + 2mp}{p} \prod_{j =0}^{m-1} \dfrac{\{ N-\alpha + 2j p \} \{ N(p-1) +\alpha -2(j +1)p \} }{p^2}  &\text{if} \,\, k=2m+1.
              \end{cases}
\end{align*}
For $\alpha \in (2+ 2(m-1)p, N)$, the Rellich type inequality
\begin{equation}\label{R_p}
 A_{k,p,\alpha}^p \int_{B_R} \frac{|u|^p}{|x|^{\alpha}} dx \le \int_{B_R} \frac{|\nabla^k u|^p}{|x|^{\alpha -kp}}\,dx
\end{equation}
holds for all $u \in C_c^{k} (B_R )$ (Ref.  \cite{Rellich,DH,Mitidieri,GGM,MOW(Rellich),ST(Rellich)}). If $p=2$, it is known that the geometric Rellich inequality
\begin{align}\label{R_p geo}
\( \prod_{j=1}^{k} \frac{jp -1}{p}  \)^{p} \int_{B_R} \frac{|u|^p}{ {\rm dist}(x, \pd B_R)^{kp}} \,dx
\le \int_{B_R} |\nabla^k u|^p\,dx
\end{align}
holds for all $u \in C_c^{k} (B_R )$ (Ref. \cite{O,B}). For the case where $p \not=2$, see the important remark in \cite{B(Math.Z)} p.879 and the end of this section.    
It is also known that both $A_{k,p}^p$ and $\( \prod_{j=1}^{k} \frac{jp -1}{p}  \)^{p}$ with $p=2$ are the optimal constants.  






\subsection{Improved Rellich inequalities on $L^2$}\label{S IR p=2}


In this subsection, we treat the case where $p=2$. 

\begin{theorem}\label{T IR p=2}

\noindent
(I) If $4-N < \alpha \le N -\gamma$ and $\gamma  >0$, then the inequality
\begin{align}\label{IR gene p=2}
\( \frac{N+\alpha -4}{4} \gamma \)^2 \int_{B_R} \frac{|u|^2}{|x|^{\alpha} \( 1-\( \frac{|x|}{R} \)^\gamma \)^{2}}\le \int_{B_R} \frac{|\lap u|^2}{|x|^{\alpha -4} }\,dx
\end{align}
holds for any functions $u \in C_{c}^{\infty}(B_R)$. Especially, when $\gamma = N-\alpha$, the constant $\( \frac{(N-\alpha)(N+\alpha -4)}{4} \)^2$ in (\ref{IR gene p=2}) is optimal and is not attained for $u \not\equiv 0$ for which the right-hand side is finite.

\noindent
(II) If 3 $\le \alpha \le \min \{ N-\gamma +2, N-3\gamma \}$ and $\gamma >0$, then the inequality
\begin{align}
\label{IR gene b p=2}
\( \frac{3}{4} \gamma^2 \)^2 \int_{B_R} \frac{|u|^2}{|x|^{\alpha} \( 1-\( \frac{|x|}{R} \)^\gamma \)^{4}} 
\le \int_{B_R} \frac{|\lap u|^2}{|x|^{\alpha -4} }\,dx
\end{align}
holds for any functions $u \in C_{c}^{\infty}(B_R)$. 
Furthermore, the constant $\( \frac{3}{4} \gamma^2 \)^2$ in (\ref{IR gene b p=2}) is optimal and is not attained for $u \not\equiv 0$ for which the right-hand side is finite.
\end{theorem}

\begin{remark}\label{Rem higher geo}
It seems difficult to show the inequality (\ref{IR gene b p=2}) with the weight $\( 1- \( \frac{|x|}{R} \)^\gamma \)^{-\beta}$ on the right-hand side. 
This is one of the reasons why we cannot show the higher order case of the inequality (\ref{IR gene b p=2}), see also Remark \ref{Rem assumption}. 
\end{remark}

Our inequalities (\ref{IR gene p=2}), (\ref{IR gene b p=2}) give improvements of the classical Rellich type inequalities (\ref{R_p}), (\ref{R_p geo}) while keeping their best constants as follows.  

\begin{corollary}\label{C IR p=2}

\noindent
(I) Let $N \ge 5$. Then the inequalities 
\begin{align*}
\int_{B_R} \frac{|u|^2}{ |x|^{4}} \,dx
\le \int_{B_R} \frac{|u|^2}{|x|^{4} \( 1- \( \frac{|x|}{R} \)^{N-4}  \)^{2}} \,dx
\le \( \frac{N(N-4)}{4} \)^{-2} \int_{B_R} |\lap u|^2 \,dx
\end{align*}
hold for any functions $u \in W_0^{2, 2} (B_R)$. 

\noindent
(II) Let $N \ge 7$. Then the inequalities
\begin{align*}
\int_{B_R} \frac{|u|^2}{ {\rm dist}(x, \pd B_R)^{4}} \,dx
\le  \int_{B_R} \frac{|u|^2}{|x|^{4} \( 1- \frac{|x|}{R}  \)^{4}} \,dx
\le \( \frac{3}{4}  \)^{-2} \int_{B_R} |\lap u|^2\,dx
\end{align*}
hold for any functions $u \in W_0^{2, 2} (B_R)$.
\end{corollary}

Furthermore, we also obtain two critical Rellich inequalities as limiting forms of our inequalities (\ref{IR gene p=2}) and (\ref{IR gene b p=2}) as $\gamma \to 0$. For the critical Rellich inequalities (\ref{ICR p=2}), (\ref{ICR  b p=2}), see e.g. \cite{CM}. 

\begin{corollary}\label{C IR p=2}
(I) If $4-N < \alpha \le N$, then the inequality
\begin{align}\label{ICR p=2}
\( \frac{N+\alpha -4}{4} \)^2 \int_{B_R} \frac{|u|^2}{|x|^{\alpha} \( \log  \frac{R}{|x|}  \)^{2}}\le \int_{B_R} \frac{|\lap u|^2}{|x|^{\alpha -4} }\,dx
\end{align}
holds for any functions $u \in C_{c}^{\infty}(B_R)$.

\noindent
(II) If 3 $\le \alpha \le N$, then the inequality
\begin{align}
\label{ICR b p=2}
\( \frac{3}{4} \)^2 \int_{B_R} \frac{|u|^2}{|x|^{\alpha} \(\log  \frac{R}{|x|} \)^{4}} 
\le \int_{B_R} \frac{|\lap u|^2}{|x|^{\alpha -4} }\,dx
\end{align}
holds for any functions $u \in C_{c}^{\infty}(B_R)$.
\end{corollary}

We shall derive the improved Rellich inequalities (\ref{IR gene p=2}), (\ref{IR gene b p=2}) simply by integration by parts and the one dimensional inequalities in \S \ref{Appendix}.


\begin{proof}(Proof of Theorem \ref{T IR p=2})
For the simplicity, we set $R=1$. We use the polar coordinate $x=r\w \,(r = |x|, \,\w \in \mathbb{S}^{N-1})$ and
\begin{align*}
\nabla u = \frac{\pd u}{\pd r} \w + \frac{1}{r} \nabla_{\mathbb{S}^{N-1}} u, \quad \lap u = \frac{\pd^2 u}{\pd r^2} + \frac{N-1}{r} \frac{\pd u}{\pd r} + \frac{\lap_{\s} u}{r^2}.
\end{align*} 

\noindent
(I) For $u \in C_c^{\infty} (B_1 \setminus \{ 0\})$, we have
\begin{align*}
\int_{B_1} \frac{|\,\lap u \,|^2}{|x|^{\alpha -4}} \,dx 
&= \int_0^1 \int_{\s} \left| \, \frac{\pd^2 u}{\pd r^2} + \frac{N-1}{r} \frac{\pd u}{\pd r} + \frac{\lap_{\s} u}{r^2} \, \right|^2 r^{N-\alpha +3} \, dr dS_{\w} \\
&= \iint \left| \, \frac{\pd^2 u}{\pd r^2} \, \right|^2 r^{N-\alpha +3} 
+ (N-1)(\alpha -3) \left| \, \frac{\pd u}{\pd r} \, \right|^2 r^{N-\alpha +1}
 + | \,\lap_{\s} u \,|^2 r^{N-\alpha -1}  \\
&- 2(N -1) \nabla_{\s} \( \frac{\pd u}{\pd r} \) \cdot (\nabla_{\s} u) \,r^{N-\alpha} -2 \nabla_{\s} \( \frac{\pd^2 u}{\pd r^2} \) \cdot (\nabla_{\s} u) \,r^{N-\alpha+1}.
\end{align*}
Since $\frac{\pd}{\pd r}$ and $\nabla_{\s}$ are commutative, we have
\begin{align*}
\int_{B_1} \frac{|\,\lap u \,|^2}{|x|^{\alpha -4}} \, dx 
&= \iint \left| \, \frac{\pd^2 u}{\pd r^2} \, \right|^2 r^{N-\alpha +3} 
+ (N-1)(\alpha -3) \left| \, \frac{\pd u}{\pd r} \, \right|^2 r^{N-\alpha +1}
 + | \,\lap_{\s} u \,|^2 r^{N-\alpha -1}  \\
&- (N -1) \frac{\pd}{\pd r} ( \,|\nabla_{\s} u |^2\,) \,r^{N-\alpha} -2 \frac{\pd}{\pd r} \(  \nabla_{\s} \( \frac{\pd u}{\pd r} \) \,\) \cdot (\nabla_{\s} u) \,r^{N-\alpha+1}\\
&= \iint \left| \, \frac{\pd^2 u}{\pd r^2} \, \right|^2 r^{N-\alpha +3} 
+ (N-1)(\alpha -3) \left| \, \frac{\pd u}{\pd r} \, \right|^2 r^{N-\alpha +1}
 + | \,\lap_{\s} u \,|^2 r^{N-\alpha -1}  \\
&+ (N -\alpha) (\alpha -2) \, | \,\nabla_{\s} u \, |^2 \,r^{N-\alpha -1} +2 \left| \,\frac{\pd}{\pd r} \( \, \nabla_{\s} u \, \) \, \right|^2 \,r^{N-\alpha+1}.
\end{align*}
If $4-N < \alpha$ and $N-\alpha  -\gamma \ge 0$, then we have the followings by Proposition \ref{Prop WHR} and Proposition \ref{Prop IWHR}.
\begin{align*}
\int_{B_1} \frac{|\,\lap u \,|^2}{|x|^{\alpha -4}} dx 
&\ge \iint \left\{ (N-1)(\alpha -3) + \( \frac{N-\alpha +2}{2} \)^2 \, \right\} \left| \, \frac{\pd u}{\pd r} \, \right|^2 r^{N-\alpha +1}
 + | \,\lap_{\s} u \,|^2 r^{N-\alpha -1}  \\
&+ \left\{  (N -\alpha) (\alpha -2) + 2 \( \frac{N-\alpha}{2} \)^2 \right\} \, | \,\nabla_{\s} u \, |^2 \,r^{N-\alpha -1} \\
&\ge \iint \( \frac{N + \alpha -4}{2} \)^2 \left| \, \frac{\pd u}{\pd r} \, \right|^2 r^{N-\alpha +1}
 + \frac{(N-\alpha)(N+\alpha -4)}{2} \, | \,\nabla_{\s} u \, |^2 \,r^{N-\alpha -1} \\
&\ge \( \frac{N + \alpha -4}{2} \)^2 \( \frac{\gamma}{2} \)^2 \iint  | u|^2 r^{N-\alpha -1} (1-r^\gamma )^{-2} \,drdS_\w \\
&= \( \frac{N + \alpha -4}{4} \gamma \)^2 \int_{B_1} \frac{|u |^2}{|x|^{\alpha} (1-|x|^\gamma )^2} \,dx 
\end{align*}
Therefore, we have the inequality (\ref{IR gene p=2}) for $u \in C_c^\infty (B_1 \setminus \{ 0\})$. By the assumption $\alpha \le N-\gamma$, we see that the inequality (\ref{IR gene p=2}) holds for $u \in C_c^\infty (B_1)$. In fact, we consider a smooth cut-off function $\phi_\ep \in C_c^\infty (B_1 \setminus \{ 0\})$ which satisfies $0 \le \phi_\ep \le 1, \phi_\ep =1$ on $B_1 \setminus B_{2\ep}, \phi_\ep =0$ on $B_{\ep}, |\nabla \phi_\ep | \le C \ep^{-1},$ and $|\lap \phi_\ep| \le C \ep^{-2}$. For $u \in C_c^\infty (B_1)$, set $u_\ep = u \phi_\ep \in C_c^\infty (B_1 \setminus \{ 0\})$. 
Thanks to the assumption $\alpha < N$, we have
\begin{align*}
\int_{B_1} \frac{|\,\lap u_\ep \,|^2}{|x|^{\alpha -4}} \, dx \to \int_{B_1} \frac{|\,\lap u \,|^2}{|x|^{\alpha -4}} \, dx, \,\,\int_{B_1} \frac{|u_\ep |^2}{|x|^{\alpha} (1-|x|^\gamma )^2} \,dx  \to \int_{B_1} \frac{|u |^2}{|x|^{\alpha} (1-|x|^\gamma )^2} \,dx,
\end{align*}
as $\ep \to 0$. Therefore, we obtain the inequality (\ref{IR gene p=2}) for $u \in C_c^\infty (B_1)$. 
Furthermore, the above calculation implies that if $U$ is an extremal function of the inequality (\ref{IR gene p=2}), then $U$ is a radially symmetric function since $|\nabla_{\s} U| =0$ a.e. in $B_1$. Also, we have $\gamma = N-\alpha$ and $U (x) = U(|x|) = C |x|^{-\frac{N- \alpha}{2}} \( 1 -|x|^{N-\alpha} \)^{\frac{1}{2}}$ a.e. in $B_1$ for some $C \in \re$ from the equality condition of Proposition \ref{Prop IWHR}. However, we have
\begin{align*}
\int_{B_1} \frac{|\,\lap U \,|^2}{|x|^{\alpha -4}} dx 
&=\( \frac{N + \alpha -4}{4} \gamma \)^2 \int_{B_1} \frac{|U |^2}{|x|^{\alpha} (1-|x|^\gamma )^2} \,dx \\
&= C^2 \( \frac{N + \alpha -4}{4} \gamma \)^2 \w_{N-1} \int_0^1 r^{-1 } (1-r^{N-\alpha} )^{-1} \,dr = \infty
\end{align*}
for any $\alpha$ and any $C \not=0$. The remaining of the proof is to show the optimality of the constant $\( \frac{N+\alpha -4}{4} \gamma \)^2$ in (\ref{IR gene p=2}). Since the test function and the calculations are completely the same as them in the proof of Theorem \ref{T IR} (I), we omit here.

\noindent
(II) For $u \in C_c^{\infty} (B_1 \setminus \{ 0\})$ and $\alpha \in [3, \min \{ N-\gamma +2, N-3\gamma \}]$, we have the followings from the above calculation and Proposition \ref{Prop IWHR}.
\begin{align*}
\int_{B_1} \frac{|\,\lap u \,|^2}{|x|^{\alpha -4}} \, dx 
&\ge \int_0^1 \int_{\s} \left| \, \frac{\pd^2 u}{\pd r^2} \, \right|^2 r^{N-\alpha +3}  + \frac{(N-\alpha)(N+\alpha -4)}{2} \, | \,\nabla_{\s} u \, |^2 \,r^{N-\alpha -1}\\
&\ge \( \frac{\gamma}{2} \)^2 \int_0^1 \int_{\s} \left| \, \frac{\pd u}{\pd r} \, \right|^2 r^{N-\alpha +1} (1-r^\gamma )^{-2} \\
&\ge \( \frac{3}{4} \gamma^2 \)^2 \int_0^1 \int_{\s} |u|^2 r^{N-\alpha -1} (1-r^\gamma )^{-4}  
=  \( \frac{3}{4} \gamma^2 \)^2 \int_{B_1} \frac{|u|^2}{|x|^{\alpha} (1-|x|^\gamma )^4} \, dx 
\end{align*}
In the same way as (I), we have the inequality (\ref{IR gene b p=2}) for $u \in C_c^{\infty} (B_1)$. Furthermore, the above calculation implies that if $U$ is an extremal function of the inequality (\ref{IR gene b p=2}), then $\alpha = 3$ and $U$ is a radially symmetric function since $|\nabla_{\s} U| =0$ a.e. in $B_1$. Also, we have $\gamma = \frac{N-3}{3}$ and $U (x) = U(|x|) = C |x|^{-\frac{N-3}{2}} \( 1 -|x|^{\frac{N-3}{3}} \)^{\frac{3}{2}}$ a.e. in $B_1$ for some $C \in \re$ from the equality condition of Proposition \ref{Prop IWHR}. However, we have
\begin{align*}
\int_{B_1} \frac{|\,\lap U \,|^2}{|x|^{\alpha -4}} dx 
&=\( \frac{3}{4} \gamma^2 \)^2 \int_{B_1} \frac{|U |^2}{|x|^{\alpha} (1-|x|^\gamma )^4} \,dx \\
&= C^2 \( \frac{3}{4} \gamma^2 \)^2 \w_{N-1} \int_0^1 r^{-1 } (1-r^{\frac{N-3}{3}} )^{-1} \,dr = \infty
\end{align*}
for any $C \not=0$. The remaining of the proof is to show the optimality of the constant $\( \frac{3}{4} \gamma^2 \)^2$ in (\ref{IR gene b p=2}).
Since the test function and the calculations are completely the same as them in the proof of Theorem \ref{T IR} (II), we omit here. 
\qed
\end{proof}


From Theorem \ref{T IR p=2} (I) and the classical Rellich type inequality (\ref{R_p}), we also obtain higher order case of the inequality (\ref{IR gene p=2}) as follows. 

\begin{corollary}\label{Cor IR k p=2}
Let $k \ge 3, k=2m+1$ or $k=2m, 1< p, \beta < \infty$, and $-2 + 4m < \alpha \le N -  \gamma$. 
Then the inequality
\begin{align}\label{IR gene k p=2}
\( \frac{\gamma}{N-\alpha} A_{k,2, \alpha} \)^2 \int_{B_R} \frac{|u|^2}{|x|^{\alpha} \( 1-\( \frac{|x|}{R} \)^\gamma \)^{2}}\le \int_{B_R} \frac{|\nabla^k u|^2}{|x|^{\alpha -2k} }\,dx
\end{align}
holds for any functions $u \in C_{c}^{k}(B_R)$. 
Furthermore, the constant $\( \frac{\gamma}{N-\alpha} A_{k,2, \alpha} \)^2$ in (\ref{IR gene k p=2}) is optimal and is not attained for $u \not\equiv 0$ for which the right-hand side is finite.
\end{corollary}

As a limiting form of (\ref{IR gene k p=2}) as $\gamma \to 0$, we obtain the critical Rellich inequality for $-2 + 4m < \alpha \le N$.
\begin{align*}
\( \frac{A_{k,2, \alpha}}{N-\alpha}  \)^2 \int_{B_R} \frac{|u|^2}{|x|^{\alpha} \( \log \frac{R}{|x|} \)^{2}} 
\le \int_{B_R} \frac{|\nabla^k u|^2}{|x|^{\alpha -2k}}\,dx.
\end{align*}

\begin{proof}(Proof of Corollary \ref{Cor IR k p=2})
Applying the Rellich type inequality (\ref{R_p}) with $\tilde{\alpha} =\alpha -4$ and $\tilde{k} = k-2$, we have
\begin{align}\label{subR_p}
A_{\tilde{k}, 2, \tilde{\alpha}}^2 \int_{B_R} \frac{|\lap u|^2}{|x|^{\tilde{\alpha}}}\,dx \le \int_{B_R} \frac{|\nabla^{\tilde{k}} (\lap u )|^2}{|x|^{\tilde{\alpha}-2\tilde{k}}}\,dx
= \int_{B_R} \frac{|\nabla^k u|^2}{|x|^{\alpha -2k} }\,dx.
\end{align}
On the other hand, from Theorem \ref{T IR p=2} (I), we have
\begin{align}\label{Thm apply}
\( \frac{N+\alpha -4}{4} \gamma \)^2 \int_{B_R} \frac{|u|^2}{|x|^{\alpha} \( 1-\( \frac{|x|}{R} \)^\gamma \)^{2}}\le \int_{B_R} \frac{|\lap u|^2}{|x|^{\alpha -4} }\,dx.
\end{align}
Note that
\begin{align}\label{A to A}
A_{\tilde{k}, 2, \tilde{\alpha}} \frac{(N-\alpha)(N+\alpha -4)}{4} = A_{k,2,\alpha}. 
\end{align}
From (\ref{subR_p}), (\ref{Thm apply}), and (\ref{A to A}), 
we obtain the inequality (\ref{IR gene k p=2}). 
\qed
\end{proof}


\subsection{Improved Rellich inequalities for radially symmetric functions on $L^p$}\label{S IR rad}

In this subsection, we treat the case where $p \not=2$ for radially symmetric functions.

\begin{theorem}\label{T IR}
Let $1 < p, \beta < \infty, \gamma > 0$, and $k \ge 2$.  

\noindent
(I) If $\alpha \le N- (p-1)\gamma$, then the inequality
\begin{align}
\label{IR gene}
\( \frac{p-1}{N-\alpha} \gamma \,A_{k,p,\alpha} \)^p \int_{B_R} \frac{|u|^p}{|x|^{\alpha} \( 1-\( \frac{|x|}{R} \)^\gamma \)^{p}}\le \int_{B_R} \frac{|\nabla^k u|^p}{|x|^{\alpha -kp} }\,dx
\end{align}
holds for any radially symmetric functions $u \in C_{c,{\rm rad}}^{k}(B_R)$.
Especially, when $\gamma = \frac{N-\alpha}{p-1}$, the constant $A_{k,p,\alpha}^p$ in (\ref{IR gene}) is optimal and is not attained for $u \not\equiv 0$ for which the right-hand side is finite.

\noindent
(II) If $\alpha \le \min \{ N -(\beta -p -1)\gamma - (N-2) p, \, N - (\beta -1)\gamma \}$, then the inequality
\begin{align}\label{IR gene b}
\( \prod_{j=0}^{k-1} \frac{\beta -jp -1}{p} \gamma \)^p \int_{B_R} \frac{|u|^p}{|x|^{\alpha} \( 1-\( \frac{|x|}{R} \)^\gamma \)^{\beta}} 
\le \int_{B_R} \frac{|\nabla^k u|^p}{|x|^{\alpha -kp} \( 1-\( \frac{|x|}{R} \)^\gamma \)^{\beta -kp}}\,dx
\end{align}
holds for any radially symmetric functions $u \in C_{c,{\rm rad}}^{k}(B_R)$. Furthermore, the constant $\( \prod_{j=0}^{k-1} \frac{\beta -jp -1}{p} \gamma \)^p$ in (\ref{IR gene b}) is optimal and is not attained for $u \not\equiv 0$ for which the right-hand side is finite.
\end{theorem}

\begin{remark}
Note that the function $\frac{\gamma}{1-r^\gamma}$ is monotone-increasing with respect to $\gamma \in (0, \frac{N-\alpha}{p-1}]$ for any $r \in (0,1)$. Therefore, in the inequality (\ref{IR gene}), we see that 
\begin{align*}
\( \frac{p-1}{N-\alpha} \gamma \,A_{k,p,\alpha} \)^p \int_{B_R} \frac{|u|^p}{|x|^{\alpha} \( 1-\( \frac{|x|}{R} \)^\gamma \)^{p}} \le A_{k,p,\alpha}^p \int_{B_R} \frac{|u|^p}{|x|^{\alpha} \( 1-\( \frac{|x|}{R} \)^{\frac{N-\alpha}{p-1}} \)^{p}} \le \int_{B_R} \frac{|\nabla^k u|^p}{|x|^{\alpha -kp} }\,dx.
\end{align*}
\end{remark}

\begin{remark}
Unfortunately, we cannnot derive the geometric Rellich inequality (\ref{R_p geo}) from our inequality (\ref{IR gene b}) due to the assumption $\alpha \le N -(\beta -p -1)\gamma - (N-2) p$. In fact, if we assume $(\alpha, \beta, \gamma) = (kp, kp, 1)$, then we have $p=1$ from the assumption. However, the case where $p =1$ is excluded in Theorem \ref{T IR}. See also the end of this section. 
\end{remark}

We easily see that our inequality (\ref{IR gene}) with $\gamma = \frac{N-\alpha}{p-1}$ gives an improvement of the classical Rellich type inequality (\ref{R_p}) on $L^p$ for radially symmetric functions. 
We also obtain the critical Rellich inequality (\ref{critical Rellich}) on $L^p$ for radially symmetric functions as a limiting form of our inequality (\ref{IR gene}) as $\gamma \to 0$. For the critical Rellich inequality (\ref{critical Rellich}) on $L^p$, see \cite{DHA,N,RS}. 

\begin{corollary}\label{C IR log}
Let $1 < p, \beta < \infty$ and $k \in \N$.
If $\alpha \le N$, then the inequality
\begin{align}\label{critical Rellich}
\( \frac{p-1}{N-\alpha}  \,A_{k,p,\alpha} \)^p \int_{B_R} \frac{|u|^p}{|x|^{\alpha} \( \log \frac{R}{|x|} \)^{p}}\le \int_{B_R} \frac{|\nabla^k u|^p}{|x|^{\alpha -kp} }\,dx
\end{align}
holds for any radially symmetric functions $u \in C_{c,{\rm rad}}^{k}(B_R)$.
\end{corollary}


To prove Theorem \ref{T IR}, we show the improved Rellich type inequality with a remainder term as follows. The remainder term comes from Theorem \ref{T IH remainder}. 

\begin{theorem}\label{T IR gene remainder}
Let $1 < p, \beta < \infty, \gamma > 0$, and $k \ge 2$.

\noindent
(I) If $\alpha \le N- (p-1)\gamma$, then the inequality
\begin{align}
\label{IR gene remainder}
\( \frac{p-1}{N-\alpha} \gamma \,A_{k,p,\alpha} \)^p \int_{B_R} \frac{|u|^p}{|x|^{\alpha} \( 1-\( \frac{|x|}{R} \)^\gamma \)^{p}}
+ \( \frac{p}{N-\alpha} A_{k,p,\alpha} \)^p \,\psi_{N,p,\alpha, p} (u)
\le \int_{B_R} \frac{|\nabla^k u|^p}{|x|^{\alpha -kp} }\,dx
\end{align}
holds for any radially symmetric functions $u \in C_{c, {\rm rad}}^{k}(B_R)$, where $\psi_{N,p,\alpha,\beta}(u)$ is given in Theorem \ref{T IH remainder}.

\noindent
(II) If $\alpha \le \min \{ N -(\beta -p -1)\gamma - (N-2) p, \, N - (\beta -1)\gamma \}$, then the inequality
\begin{align}\label{IR gene b remainder}
\( \prod_{j=0}^{k-1} \frac{\beta -jp -1}{p} \gamma \)^p \int_{B_R} \frac{|u|^p}{|x|^{\alpha} \( 1-\( \frac{|x|}{R} \)^\gamma \)^{\beta}} 
&+ \( \prod_{j=1}^{k-1} \frac{\beta -jp -1}{p} \gamma \)^p \psi_{N,p,\alpha, \beta} (u) \notag \\
&\le \int_{B_R} \frac{|\nabla^k u|^p}{|x|^{\alpha -kp} \( 1-\( \frac{|x|}{R} \)^\gamma \)^{\beta -kp}}\,dx
\end{align}
holds for any radially symmetric functions $u \in C_{c,{\rm rad}}^{k}(B_R)$, where $\psi_{N,p,\alpha,\beta}(u)$ is given in Theorem \ref{T IH remainder}.
\end{theorem}

Concerning the classical Rellich inequality (\ref{R_p}), we need the assumption where $\alpha > 2 + 2(m-1)p$. In fact, if there is no restriction with respect to $\alpha < N$, it is possible that the best constant of (\ref{R_p}) for any functions is less than $A_{k,p, \alpha}^p$. At least when $p=2$, this is true. Namely, symmetry breaking phenomenon occurs for some $\alpha$. For the details, see \cite{CM,TZ}. However, the best constant of (\ref{R_p}) for radially symmetric functions is $A_{k,p,\alpha}^p$ for any $\alpha < N$, see e.g. \cite{M(2014)} or Corollary \ref{Cor H 0 to k} in \S \ref{Appendix}. 
This is the reason why we do not need any restrictions with respect to $\alpha < N$ in Theorems in this section.


\begin{proof}(Proof of Theorem \ref{T IR gene remainder}) 
(I) We obtain the following from Theorem \ref{T IH remainder} with $\beta = p$ and (\ref{App HR 1 to 2}) in Proposition \ref{Prop WHR}.
\begin{align*}
\( \frac{p-1}{p} \gamma \)^p \int_{B_R} \frac{|u|^p}{|x|^{\alpha} \( 1-\( \frac{|x|}{R} \)^\gamma \)^{p}}
+  \psi_{N,p,\alpha, p} (u)
&\le \int_{B_R} \frac{|\nabla u|^p}{|x|^{\alpha -p} }\,dx \\
&\le \left| \, \frac{N(p-1) + \alpha -2p}{p} \, \right|^{-p} \int_{B_R} \frac{|\lap u|^p}{|x|^{\alpha - 2p} }\,dx
\end{align*}
This implies the inequality (\ref{IR gene remainder}) in the case where $k=2$. 
Combining this with Corollary \ref{Cor H 0 to k}, we have
\begin{align*}
&\( \frac{p-1}{N-\alpha} \gamma \,A_{2,p,\alpha} \)^p \int_{B_R} \frac{|u|^p}{|x|^{\alpha} \( 1-\( \frac{|x|}{R} \)^\gamma \)^{p}}
+ \left| \, \frac{N(p-1) + \alpha -2p}{p} \, \right|^{p} \psi_{N,p,\alpha, p} (u) \\
&\le  \int_{B_R} \frac{|\lap u|^p}{|x|^{\alpha - 2p} }\,dx 
\le A_{k-2, p, \alpha -2p}^{-p} \int_{B_R} \frac{|\nabla^{k-2} \lap u|^p}{|x|^{\alpha -2p -(k-2)p} }\,dx 
= A_{k-2, p, \alpha -2p}^{-p} \int_{B_R} \frac{|\nabla^{k} u|^p}{|x|^{\alpha -kp} }\,dx.
\end{align*}
Since $A_{2,p,\alpha} \,A_{k-2, p, \alpha -2p} = A_{k,p,\alpha}$, 
the inequality (\ref{IR gene remainder}) holds for any $k$. 

\noindent
(II) Let $\alpha \le \min \{ N -(\beta -p -1)\gamma - (N-2) p, \, N - (\beta -1)\gamma \}$. Then we obtain the following from Theorem \ref{T IH remainder} and Corollary \ref{Cor IHR 1 to 2}.
\begin{align*}
\( \frac{\beta -1}{p} \gamma \)^p &\int_{B_R} \frac{|u|^p}{|x|^{\alpha} \( 1-\( \frac{|x|}{R} \)^\gamma \)^{\beta}} 
+ \psi_{N,p,\alpha, \beta} (u) \\
&\le \int_{B_R} \frac{|\nabla u|^p}{|x|^{\alpha -p} \( 1-\( \frac{|x|}{R} \)^\gamma \)^{\beta -p}}\,dx\\
&\le \( \frac{\beta -p-1}{p} \gamma \)^{-p} \int_{B_R} \frac{|\lap u|^p}{|x|^{\alpha -2p} \( 1-\( \frac{|x|}{R} \)^\gamma \)^{\beta -2p}}\,dx.
\end{align*}
This implies the inequality (\ref{IR gene b remainder}) in the case where $k=2$. 
Note that $\alpha -jp \le N -(\beta - jp -1)\gamma - (N-1) p$ and $\alpha - jp \le N - (\beta -jp -1) \gamma$ hold for any integer $j \ge 1$. Assume that the inequality (\ref{IR gene b remainder}) holds for $k \ge 2$. Then, from Corollary \ref{Cor IHR 1 to 2} or Theorem \ref{T IH}, we have
\begin{align*}
\( \prod_{j=0}^{k} \frac{\beta -jp -1}{p} \gamma \)^p &\int_{B_R} \frac{|u|^p}{|x|^{\alpha} \( 1-\( \frac{|x|}{R} \)^\gamma \)^{\beta}} 
+ \( \prod_{j=1}^{k} \frac{\beta -jp -1}{p} \gamma \)^p \psi_{N,p,\alpha, \beta} (u) \notag \\
&\le \( \frac{\beta -kp -1}{p} \gamma \)^p \int_{B_R} \frac{|\nabla^k u|^p}{|x|^{\alpha -kp} \( 1-\( \frac{|x|}{R} \)^\gamma \)^{\beta -kp}}\,dx\\
&\le \int_{B_R} \frac{|\nabla^{k+1} u|^p}{|x|^{\alpha -(k+1)p} \( 1-\( \frac{|x|}{R} \)^\gamma \)^{\beta -(k+1)p}}\,dx.
\end{align*}
Therefore the inequality (\ref{IR gene b remainder}) holds for $k+1$. Hence, we see that the inequality (\ref{IR gene b remainder}) holds for any $k \ge 2$.
\qed
\end{proof}


\begin{proof}(Proof of Theorem \ref{T IR}) For the simplicity, we set $R=1$. 

\noindent
(I) Let $\gamma = \frac{N-\alpha}{p-1}$. 
To show the optimality of the constant $A_{k,p,\alpha}^p$ in (\ref{IR gene}), we consider the test function
\begin{align*}
	g_B (x)= \psi_\delta (x) \,|x|^{-B},
\end{align*}
for $B < \frac{N -\alpha}{p}$ and small $\delta >0$, 
where $\psi_\delta$ is a smooth radially symmetric function which satisfies $0 \le \psi_\delta \le 1, \psi_\delta \equiv 0$ on $B_1 \setminus B_{2\delta}$ and $\psi_\delta \equiv 1$ on $B_{\delta}$. 
Since
\begin{align*}
| \nabla^k  |x|^B  | = |x|^{A-k} 
\begin{cases}
\vspace{0.5em}
\left| \prod_{j=0}^{m-1} (A-2j) (N+A-2j-2) \,\right| \quad &\text{if}\,\, k=2m,\\
\left| (A-2m) \, \prod_{j=0}^{m-1} (A-2j) (N+A-2j-2) \,\right| &\text{if}\,\, k=2m+1,
\end{cases}
\end{align*}
we have
\begin{align*}
A_{k,p,\alpha}^p
\le \frac{\int_{B_1} \frac{| \nabla^k g_B |^p}{|x|^{\alpha -kp}} \,dx}{\int_{B_1} \frac{|g_B|^p}{|x|^\alpha \( 1-|x|^\gamma \)^p} \,dx} 
\le A_{k,p,\alpha}^p + o(1)
\quad \( B \to \frac{N-\alpha}{p},\,\, \delta \to 0 \).
\end{align*}
Therefore the constant $A_{k,p,\alpha}^p$ in (\ref{IR gene}) is optimal. 
Since there exists the nonnegative remainder term in Theorem \ref{T IR gene remainder} (I), we observe that if there exists an extremal function $U=U(x)$ of the inequality (\ref{IR gene}), then $U(x)= c \( |x|^{-\frac{N-\alpha}{p-1}} -1 \)^{\frac{p -1}{p}} = c |x|^{-\frac{N-\alpha}{p}} \(  1-|x|^{\frac{N-\alpha}{p-1}} \)^{\frac{p -1}{p}}$ for some $c \in \re$. However, if $c \not= 0$, then the right-hand side of (\ref{IR gene}) diverges since
\begin{align*}
\int_{B_\ep } \frac{| \nabla^k U |^p}{|x|^{\alpha -kp} } \,dx
&\ge C(\ep) \int_{B_\ep} \frac{\left| \nabla^k \,|x|^{-\frac{N-\alpha}{p}} \right|^p}{|x|^{\alpha -kp} } \,dx + D(\ep)\\
&\ge \tilde{C}(\ep) \int_{B_\ep} |x|^{-N} \,dx + D(\ep) =\infty
\end{align*}
for any small $\ep >0$, where $C(\ep), \tilde{C}(\ep) \not=0, D(\ep)$ are some constants depending on $\ep$.

\noindent
(II) Let $\alpha \le N - (\beta -p -1)\gamma - (N-2)p$. 
We show the optimality of the constant $\( \prod_{j=0}^{k-1} \frac{\beta -jp -1}{p} \gamma \)^p$ in (\ref{IR gene b}). For $A > \frac{\beta-1}{p}$ and small $\delta >0$, set
\begin{align*}
	f_A (x)= \phi_\delta (x) \,(1-|x|^\gamma)^A,
\end{align*}
where $\phi_\delta$ is a smooth radially symmetric function which satisfies $0 \le \phi_\delta \le 1, \phi_\delta \equiv 0$ on $B_{1-2\delta}$ and $\phi_\delta \equiv 1$ on $B_1 \setminus B_{1-\delta}$. 
Note that
\begin{align*}
| \nabla^k (1-|x|^\gamma)^A | &= \left| \,\sum_{j=1}^k C_{k,j} \,|x|^{j\gamma -k} (1-|x|^\gamma )^{A-j} \,\right| \\
&= | C_{k,k} | \,|x|^{k(\gamma -1)} (1-|x|^\gamma )^{A-k} + o\( (1-|x|^\gamma )^{A-k} \) \quad (|x| \to 1),
\end{align*}
where $C_{k,k} = (-\gamma)^k \prod_{\ell =1}^k (A-\ell +1)$, see Proposition \ref{prop log} in \S \ref{Appendix}.
Then we have
\begin{align*}
\( \prod_{j=0}^{k-1} \frac{\beta -jp -1}{p} \gamma \)^p 
&\le \frac{\int_{B_1} \frac{| \nabla^k f_A |^p}{|x|^{\alpha -kp} \( 1-|x|^\gamma \)^{\beta -kp}} \,dx}{\int_{B_1} \frac{|f_A|^p}{|x|^\alpha \( 1-|x|^\gamma \)^\beta} \,dx} \\
&\le \frac{\gamma^{kp} \left| \prod_{\ell =1}^k (A-\ell +1) \right|^p \int_{1-\delta}^1 \( 1-r^\gamma \)^{Ap-\beta} r^{N-1-\alpha + k\gamma p}\,dr + o(1)
}{\int_{1-\delta}^{1} \( 1-r^\gamma \)^{Ap-\beta} r^{N-1-\alpha}\,dr} \\
&= \( \prod_{j=0}^{k-1} \frac{\beta -jp -1}{p} \gamma \)^p + o(1) \quad \( A \to \frac{\beta-1}{p},\,\, \delta \to 0 \).
\end{align*}
Therefore the constant $\( \prod_{j=0}^{k-1} \frac{\beta -jp -1}{p} \gamma \)^p$ in (\ref{IR gene b}) is optimal.

The proof of Theorem \ref{T IR} is now complete.
\qed
\end{proof}

In the end of this section, we formulate the following three conjectures for the general case   (Ref. \cite{B(Math.Z)} p.879):

\begin{conjecture}
Let $k \ge 2, 1 < p < \frac{N}{k}$, and 
$0 < \gamma \le \frac{N-kp}{kp -1}$. Then the inequality
\begin{align*}
\( \prod_{j=1}^{k} \frac{j p -1}{p} \gamma \)^p \int_{B_R} \frac{|u|^p}{|x|^{kp} \( 1-\( \frac{|x|}{R} \)^\gamma \)^{kp}} 
\le \int_{B_R} |\nabla^k u|^p \,dx
\end{align*}
holds for any functions $u \in C_{c}^{k}(B_R)$. Furthermore, the constant $\( \prod_{j=1}^{k} \frac{j p -1}{p} \gamma \)^p$ is optimal and is not attained for $u \not\equiv 0$ for which the right-hand side is finite.
\end{conjecture}

\begin{conjecture}
Let $1 < p < \infty$ and $k \ge 2$. Then the inequality
\begin{align*}
\( \prod_{j=1}^{k} \frac{j p -1}{p} \)^p \int_{B_R} \frac{|u|^p}{{\rm dist} (x, \pd B_R)^{kp}} 
\le \int_{B_R} |\nabla^k u|^p \,dx
\end{align*}
holds for any functions $u \in C_{c}^{k}(B_R)$. Furthermore, the constant $\( \prod_{j=1}^{k} \frac{j p -1}{p} \)^p$ is optimal and is not attained for $u \not\equiv 0$ for which the right-hand side is finite.
\end{conjecture}

\begin{conjecture}
Let $N > k \ge 2$. Then the inequality
\begin{align*}
\( \prod_{j=1}^{k} \frac{j N -k}{N} \)^{\frac{N}{k}} \int_{B_R} \frac{|u|^{\frac{N}{k}}}{|x|^N \( \log \frac{R}{|x|} \)^{N}} 
\le \int_{B_R} |\nabla^k u|^{\frac{N}{k}} \,dx
\end{align*}
holds for any functions $u \in C_{c}^{k}(B_R)$. Furthermore, the constant $\( \prod_{j=1}^{k} \frac{j N -k}{N} \)^{\frac{N}{k}}$ is optimal and is not attained for $u \not\equiv 0$ for which the right-hand side is finite.
\end{conjecture}

%
%

\section{Appendix}\label{Appendix}
 
First, we list the well-known one dimensional Hardy type inequality (\ref{App zero O}) and the Hardy-Rellich inequality (\ref{App HR 1 to 2}) for radially symmetric functions including their proofs.

\begin{proposition}\label{Prop WHR}
The following inequalities hold:
\begin{align}\label{App zero O}
&\left| \frac{a+1-p}{p} \right|^p \int_0^R r^{a-p} |w(r) |^p\,dr 
\le \int_0^R r^a | w'(r)|^p \,dr  \\
&\text{for any}\,w \in C^1 (0, R)\,\text{with}\,\lim_{r \to 0} r^{a+1-p} |w(r)|^p =0=\lim_{r \to R -0} w(r),\, \text{where}\,\, a\in \re,\, p\ge1.\notag \\
\label{App HR 1 to 2}
&\left| \frac{N(p-1) + \alpha -p}{p} \right|^p \int_{B_R} \frac{|\nabla u|^p}{|x|^\alpha} \,dx 
\le \int_{B_R} \frac{|\lap u|^p}{|x|^{\alpha -p}} \,dx\\
&\text{for any}\,u \in C_{c,{\rm rad}}^2 (B_R), \,\text{where}\,\, p \ge 1, \alpha <N.\notag
\end{align}
\end{proposition}


\begin{proof}
If $a+1-p = 0$, then the inequality (\ref{App zero O}) is trivial. Therefore, we assume that $a+1-p \not= 0$
Since $\lim_{r \to 0} r^{a+1-p} |w(r)|^p =0=w(R)$, we have
\begin{align*}
\int_0^R r^{a-p} |w(r) |^p\,dr 
&= \left[ \frac{r^{a+1-p} \,|w(r) |^p}{a+1-p} \right]_0^R
- \frac{p}{a+1-p} \int_0^R  r^{a+1-p} |w|^{p-2}w w' \,dr \\
&\le \left| \frac{p}{a+1-p} \right| \( \int_0^R r^{a-p} |w(r) |^p\,dr  \)^{\frac{p-1}{p}} \( \int_0^R r^{a} |w'(r) |^p\,dr  \)^{\frac{1}{p}}
\end{align*}
which implies (\ref{App zero O}). Next, we shall show (\ref{App HR 1 to 2}) for $u \in C_{c,{\rm rad}}^2 (B_R)$. By applying (\ref{App zero O}) for $w (r)=r^{N-1}u'(r)$ and $a=-1+p -N(p-1) -\alpha +p$, we have
\begin{align*}
\int_{B_R} \frac{|\lap u|^p}{|x|^{\alpha -p}} \,dx
&= \w_{N-1} \int_0^R r^a | w'(r)|^p \,dr \\
&\ge \w_{N-1} \left| \frac{a+1-p}{p} \right|^p \int_0^R r^{a-p} |w(r) |^p\,dr \\
&= \left| \frac{N(p-1) + \alpha -p}{p} \right|^p \int_{B_R} \frac{|\nabla u|^p}{|x|^\alpha} \,dx.
\end{align*}
Here note that $\lim_{r \to 0} r^{a+1-p} |w(r)|^p = \lim_{r \to 0} r^{N-\alpha} |u'(r)|^p =0$ since $\alpha < N$. 
Therefore we obtain (\ref{App HR 1 to 2}). 
\qed
\end{proof}

As a Corollary, we also obtain the following inequality.

\begin{corollary}\label{Cor H 0 to k}(Ref. \cite{M(2014)})
Let $\alpha < N$. Then the inequality
\begin{align*}
A_{k,p,\alpha}^p \int_{B_R} \frac{|u|^p}{|x|^\alpha} \,dx 
\le \int_{B_R} \frac{|\nabla^k u|^p}{|x|^{\alpha -kp}} \,dx 
\end{align*} 
holds for any radially symmetric functions $u \in C_{c,{\rm rad}}^k (B_R)$.
\end{corollary}


Next, we show one dimensional improved Hardy inequality (\ref{1dim IWH}). Although the proof is essentially the same as it of Theorem \ref{T IH remainder 2}, we give the proof here again.  

\begin{proposition}\label{Prop IWHR}
Let $\gamma >0, \,b < -1, a+1 + (b+1)\gamma \ge 0,$ and $1\le p< \infty$. 
The inequality
\begin{align}\label{1dim IWH}
\( \,-\frac{b+1}{p} \gamma \,\)^p \int_0^R r^{a} \( 1- \( \frac{r}{R} \)^\gamma \)^b |w(r) |^p\,dr 
\le \int_0^R r^{a+p} \( 1- \( \frac{r}{R} \)^\gamma \)^{b+p} | w'(r)|^p \,dr   
\end{align}
holds for any $w \in C^1 (0, R)\,\text{with}\lim_{r \to R-0} \( R- r \)^{b+1} |w(r)|^p = \lim_{r \to +0} r^{a+1} |w(r)|^p =0$. Furthermore, if we assume that the equality of (\ref{1dim IWH}) is attained by some non-zero function $w$, then $a+1 + (b+1)\gamma = 0$ and $w(r)= C\( \( \frac{r}{R} \)^{-\gamma} -1 \)^{-\frac{b+1}{p}}$ a.e. in $r \in (0,R)$ for some constant $C \not= 0$, whose integrals on both sides in (\ref{1dim IWH}) diverge.
\end{proposition}

\begin{proof}
For the simplicity, we set $R=1$. Then we have
\begin{align*}
&-(b+1) \gamma \int_0^1 r^{a} ( 1- r^\gamma )^b |w(r) |^p\,dr 
= \int_0^1 r^{a+1 +(b+1)\gamma} \left[ ( r^{-\gamma} -1 )^{b+1} \right]' |w(r) |^p\,dr \\
&=-p \int_0^1 r^{a+1} ( 1-r^{\gamma} )^{b+1} |w(r) |^{p-2} w(r) w'(r) \,dr - (a+1 +(b+1)\gamma) \int_0^1 r^{a} ( 1- r^\gamma )^{b+1} |w(r) |^p\,dr 
\end{align*}
On the last equality, we used $\lim_{r \to 1-0} \( 1- r \)^{b+1} |w(r)|^p =0 = \lim_{r \to +0} r^{a+1} |w(r)|^p =0$. 
Since $a+1 + (b+1)\gamma \ge 0$, we have
\begin{align}\label{Holder}
&\int_0^1 r^{a} ( 1- r^\gamma )^b |w(r) |^p\,dr \notag\\
&\le \int_0^1 r^{a+1} ( 1-r^{\gamma} )^{b+1} |w(r) |^{p-2} w(r) \( \frac{p}{(b+1)\gamma} w'(r) \) \,dr \notag \\
&\le \( \frac{p}{(b+1)\gamma} \) \( \int_0^1 r^{a} ( 1- r^\gamma )^b |w(r) |^p\,dr  \)^{\frac{p-1}{p}} \( \int_0^1 r^{p+a} ( 1- r^\gamma )^{p+b} |w'(r) |^p\,dr  \)^{\frac{1}{p}} 
\end{align}
which implies (\ref{1dim IWH}). Assume that there exists an extremal function $w \not= 0$ which attains the equality of (\ref{1dim IWH}). Then we see that $a+1 + (b+1)\gamma = 0$ from the above calculations. Since $|w|$ is also an extremal function, we may assume that $w$ is nonnegative. Besides, if there exists a interval $(A, B) \subset (0,1)$ such that $w(r) < w(A) = w(B)$ for any $r \in (A, B)$, set $\tilde{w} (r) = w(r)$ for $r \in (0,1) \setminus [A,B]$ and $\tilde{w} (r) = w(A) = w(B)$ for $r \in (A, B)$. Then we see that
\begin{align*}
\int_0^1 r^{a} \( 1- r^\gamma \)^b |w(r) |^p\,dr 
&< \int_0^1 r^{a} \( 1- r^\gamma \)^b |\tilde{w}(r) |^p\,dr, \\ 
 \int_0^1 r^{a+p} \( 1- r^\gamma \)^{b+p} | w'(r)|^p \,dr   
&> \int_0^1 r^{a+p} \( 1- r^\gamma \)^{b+p} | \tilde{w}'(r)|^p \,dr
\end{align*}
which contradicts that $w$ is an extremal function. Therefore, we may assume that $w$ is nonnegative and non-increasing. 
From the equality condition of the H\"older inequality in (\ref{Holder}), 
\begin{align*}
w' (r) = \frac{(b+1) \gamma}{p} \frac{w(r)}{r \( 1- r^\gamma\)}
\end{align*}
which implies that $w(r)= C\( r^{-\gamma} -1 \)^{-\frac{b+1}{p}}$ a.e. in $r \in (0,R)$ and for some constant $C > 0$, where $\gamma = - \frac{a+1}{b+1}$. Furthermore, we can check that
\begin{align*}
\int_0^1 r^{a+p} \( 1- r^\gamma \)^{b+p} | w'(r)|^p \,dr 
&= \( \frac{(b+1)}{p} \gamma \)^p\int_0^1 r^{a} \( 1- r^\gamma \)^b |w(r) |^p\,dr\\
&=\( \frac{(b+1)}{p} \gamma \)^p C^p \int_0^1 r^{-1} \( 1- r^\gamma \)^{-1 } \,dr = \infty.
\end{align*}
\qed
\end{proof}

As a corollary of Proposition \ref{Prop IWHR}, we can obtain the following in the same way as the proof of (\ref{App HR 1 to 2}).

\begin{corollary}\label{Cor IHR 1 to 2}
Let $p \ge 1, \beta >1, \gamma >0$, and $\alpha \le N - (\beta -1)\gamma - (N-1)p$. Then the inequality
\begin{align}\label{IHR 1 to 2}
\left( \frac{\beta -1}{p} \gamma \right)^p \int_{B_R} \frac{|\nabla u|^p}{|x|^\alpha \( 1- \( \frac{|x|}{R} \)^\gamma \)^\beta} \,dx 
\le \int_{B_R} \frac{|\lap u|^p}{|x|^{\alpha -p} \( 1- \( \frac{|x|}{R} \)^\gamma \)^{\beta -p}} \,dx
\end{align}
holds for any $u \in C_{c,{\rm rad}}^2 (B_R)$. 
\end{corollary}

\begin{remark}\label{Rem assumption}
Due to the assumption $\alpha \le N - (\beta -1)\gamma - (N-1)p$ in the inequality (\ref{IHR 1 to 2}), we cannot show $L^p (p\not= 2)$ version of the inequality (\ref{IR gene b p=2}) in the same way as the proof in \S \ref{S IR rad}.
\end{remark}


Finally, we calculate $\nabla^k \left[ \( 1- |x|^\gamma \)^A \, \right]$ to show Theorems in \S \ref{Sec Higher}.

\begin{proposition}\label{prop log}
Let $k, m \in \N$. 
\begin{align}\label{log e}
&\nabla^k \left[ \( 1- |x|^\gamma \)^A \, \right] 
=  \sum_{j=1}^{k} C_{k, j}  \,
|x|^{j\gamma -k} (1-|x|^\gamma )^{A-j} 
\begin{cases}
1 \quad &\text{if} \,\,\, k=2m,\\
\frac{x}{|x|} &\text{if} \,\,\, k=2m+1,
\end{cases}
\end{align}
where $\{ \, C_{k,j}\, \}_{j=1}^k$ is given by 
\begin{align*}
C_{k,k} &= (-\gamma)^k \prod_{\ell =1}^k (A-\ell +1) \,\, (k \ge 1),\\
C_{k,1} &= 
\begin{cases}
-A \prod_{\ell =1}^m \, (\gamma -2\ell +2) (N+\gamma -2\ell) \quad &\text{if}\,\, k=2m,\\
-A (\gamma -2m) \prod_{\ell =1}^m \, (\gamma -2\ell +2) (N+\gamma -2\ell) &\text{if}\,\, k=2m+1,
\end{cases}
\,\, (k \ge 2)\\
C_{k,j} &= \begin{cases}
C_{k-1, j} (N+j\gamma -k) -C_{k-1, j-1} (A-j+1) \quad &\text{if}\,\, k=2m,\\
C_{k-1, j} (j\gamma -k+1) -C_{k-1, j-1} (A-j+1) \gamma &\text{if}\,\, k=2m+1,
\end{cases}\\
&{}\hspace{20em} (j=2, \cdots, k-1, \,k \ge 3)
\end{align*}
\end{proposition}

\begin{proof}
Since
\begin{align*}
&\nabla \left[  \( 1-|x|^\gamma \)^A \right]
= -A\gamma \,|x|^{\gamma -2} x \,(1-|x|^\gamma )^{A-1}, \\
&\lap \left[ \( 1-|x|^\gamma \)^A \right] 
= -A\gamma  (N+\gamma -2) \,|x|^{\gamma -2}  \,(1-|x|^\gamma )^{A-1}
+ A (A-1) \gamma^2 \,|x|^{2\gamma -2}  \,(1-|x|^\gamma )^{A-2},
\end{align*}
we see that (\ref{log e}) holds when $k=1, 2$. Now we assume that (\ref{log e}) holds for $k=2m \ge 2$. Then we have
\begin{align*}
&\nabla^{k+1} \left[ \( 1-|x|^\gamma \)^A \right] 
= \nabla \lap^{m} \left[ \( 1-|x|^\gamma \)^A \right] 
= \sum_{j=1}^{k} C_{k, j}  \,
\nabla \left[ \,|x|^{j\gamma -k} (1-|x|^\gamma )^{A-j} \,\right] \\
&= \sum_{j=1}^{2m } C_{2m, j} \,(j\gamma -2m)\, |x|^{j\gamma -2m -2} x\, (1-|x|^\gamma )^{A-j} - C_{2m, j} \,\gamma (A-j) \, |x|^{j\gamma +\gamma -2m -2} x\, (1-|x|^\gamma )^{A-j-1} \\
&= C_{2m, 1} \,(\gamma -2m)\, |x|^{\gamma - (k+1)}\, (1-|x|^\gamma )^{A-1} \( \frac{x}{|x|} \)\\
&+ \sum_{j=2}^{2m} \left\{ C_{2m,j} \,(j\gamma -2m)  - C_{2m, j-1} \,\gamma (A-j+1) \right\} |x|^{j\gamma - (k+1)} \, (1-|x|^\gamma )^{A-j} \( \frac{x}{|x|} \) \\
&-C_{2m, 2m} \,\gamma \,(A -2m)\, |x|^{(k+1)\gamma - (k+1)}\, (1-|x|^\gamma )^{A-(k+1)} \( \frac{x}{|x|} \)\\
&=\sum_{j=1}^{k+1} C_{k+1, j}  \,|x|^{j\gamma -(k+1)} (1-|x|^\gamma )^{A-j} \( \frac{x}{|x|} \) 
\end{align*}
which implies that (\ref{log e}) holds for $k+1$. 
In the same way, we can also show (\ref{log e}) in the case where $k=2m +1$. 
Thus (\ref{log e}) holds for any $k \in \N$. 
\qed
\end{proof}


\begin{acknowledgements}
The author thanks to Prof. T. Ozawa (Waseda University) for his advice. 
This work was partly supported by Osaka City University Advanced
Mathematical Institute (MEXT Joint Usage/Research Center on Mathematics
and Theoretical Physics JPMXP0619217849). 
The author was supported by JSPS KAKENHI Early-Career Scientists, No. JP19K14568.
\end{acknowledgements}

%
%



\end{document}